\documentclass[12pt]{amsart}

\usepackage{graphicx,color,rotating}
 \usepackage{amsmath}
  \usepackage{amssymb,bbm}
  \usepackage{mathrsfs}
  \usepackage{amscd}
\usepackage{enumerate}
\usepackage{multirow}
\usepackage{hyperref}
\usepackage{tikz}
\usetikzlibrary{arrows}
\usepackage{comment}

\definecolor{darkred}{rgb}{.75,0,0}
\definecolor{darkgreen}{rgb}{0, .6, 0}

\newcommand{\be}{\begin{equation}}
\newcommand{\ee}{\end{equation}}
\newcommand{\defi}[1]{\textbf{#1}}
\newcommand{\ds}{\displaystyle}
\newcommand{\iden}{\mathbbm{1}}
\newcommand{\id}{\varepsilon}

\newcommand{\landslide}{Landslide sandpile model}
\newcommand{\In}{\mathrm{In}}
\newcommand{\LL}{{\mathscr L}}

\newcommand{\ov}[1]{\ensuremath{\overline {#1}}}
\newcommand{\Out}{\mathrm{Out}}
\renewcommand{\P}{\mathbb{P}}

\newcommand{\rootnode}{{\mathsf{r}}}
\newcommand{\su}{\mathop{\boldsymbol{s}}}

\newcommand{\source}{\sigma}
\newcommand{\tree}{\mathcal{T}}
\newcommand{\trickle}{Trickle-down sandpile model}
\newcommand{\wreath}{\gamma}

\newcommand{\Z}{\mathbb{Z}}

\newcommand{\dom}{\unlhd}
\newcommand{\RR}{{\mathscr R}}
\newcommand{\JJ}{{\mathscr J}}

\newcommand{\suchthat}{\mid}

\makeatletter
\newskip\@bigflushglue \@bigflushglue = -100pt plus 1fil

\def\bigcentering{\let\\\@centercr\rightskip\@bigflushglue%
\leftskip\@bigflushglue
\parindent\z@\parfillskip\z@skip}

\makeatother

\newtheorem{thm}{Theorem}[section]
\newtheorem{cor}[thm]{Corollary}
\newtheorem{lem}[thm]{Lemma}
\newtheorem{prop}[thm]{Proposition}
{\theoremstyle{definition}
}
{\theoremstyle{remark}
\newtheorem{rem}[thm]{Remark}}
\newtheorem{conj}[thm]{Conjecture}
{\theoremstyle{remark}
\newtheorem{eg}[thm]{Example}}
{\theoremstyle{remark}
}
\numberwithin{equation}{section}

\begin{document}

\title{Directed Nonabelian Sandpile Models on Trees}

\author[A. Ayyer]{Arvind Ayyer}
\curraddr{Department of Mathematics, Department of Mathematics, Indian Institute of Science, Bangalore - 560012, India.}
\email{arvind@math.iisc.ernet.in}

\address{Department of Mathematics, UC Davis, One Shields Ave., Davis, CA 95616-8633, U.S.A.}

\author[A. Schilling]{Anne Schilling}
\address{Department of Mathematics, UC Davis, One Shields Ave., Davis, CA 95616-8633, U.S.A.}
\email{anne@math.ucdavis.edu}

\author[B. Steinberg]{Benjamin Steinberg}
\address{Department of Mathematics, City College of New York, Convent Avenue at 138th Street,
New York, NY 10031, U.S.A.}
\email{bsteinberg@ccny.cuny.edu}

\author[N. M. Thi\'ery]{Nicolas M.~Thi\'ery}
\address{Univ Paris-Sud, Laboratoire de Math\'ematiques d'Orsay,
  Orsay, F-91405; CNRS, Orsay, F-91405, France}
\email{Nicolas.Thiery@u-psud.fr}

\begin{abstract}
We define two general classes of nonabelian sandpile models on
directed trees (or arborescences), as models of nonequilibrium
statistical physics. 
Unlike usual applications of the well-known abelian sandpile model,
these models have the property that sand grains
can enter only through specified reservoirs.

In the \trickle, sand grains are allowed to move one at a time.  For
this model, we show that the stationary distribution is of product
form.  In the \landslide, all the grains at a vertex topple at once,
and here we prove formulas for all eigenvalues, their multiplicities,
and the rate of convergence to stationarity.  The proofs use wreath
products and the representation theory of monoids.
\end{abstract}

\date{\today}

\maketitle

\section{Introduction}

Abelian sandpile models (ASMs) form one of the best understood classes of
models in statistical physics motivated by the problem of
understanding self-organized criticality~\cite{btw1987}. They can be
defined for any graph, directed or otherwise.  The models are
stochastic and simple to describe.  At any given time, each vertex of
the graph contains a certain number of grains of sand less than its
degree (outdegree in case of a directed graph).  At each time step, a
grain of sand is added to a random vertex.  If the number of grains is
still less than its degree, this is the new configuration.  On the
other hand if, as a result, the number of grains at that vertex
becomes more than its degree, then the vertex is said to be unstable.
It then topples, giving one grain to each of its neighbors along the
edges. If more vertices become unstable as a result, they too topple.
The model is defined by generators which describe the toppling for
each vertex.  The remarkable property of the abelian sandpile model is
that these generators commute. This makes the models particularly
amenable to computations of interest to physicists such as the
distribution of avalanches. For physically motivated reviews of
self-organized criticality and the abelian sandpile model,
see~\cite{dhar1990,ip1998,dhar1999}.

The model was also introduced around the same time by mathematicians under the 
name of chip-firing-games on graphs~\cite{bls1991}. ASMs on a graph are naturally 
related to other structures on the graph such as its sandpile group (also known as its
critical group) \cite{biggs1999}, spanning trees and the Tutte
polynomial~\cite{coriborgne2003}. For mathematically oriented reviews
of the abelian sandpile model, see~\cite{hlmppw2008}
and~\cite[Appendix]{postnikov.shapiro.2004}.

Nonequilibrium statistical physics largely deals with the study of systems 
in contact with (infinite) reservoirs. The classical example of such a system is a metal
bar, both of whose ends are kept at different temperatures by means of baths (i.e. reservoirs). 
Although it is clear that heat will flow from the higher temperature reservoir to the lower one, 
specific statistical properties are not known. In fact, very few universal laws are known for such 
systems. It is therefore of great interest to understand toy examples of such dynamical systems 
in detail. We will model such systems by irreversible Markov chains, where there is a clear 
direction of the flow of particles.

We are interested in understanding real {\em finite} systems which interact with reservoirs only 
at the boundary, such as the metal bar example above. By that, we mean that we would not 
only like to understand the stationary distribution of these models, but also transient quantities, 
such as the time they take to reach the stationary distribution. We would also like our model to 
be {\em generic} in the sense that hopping rates for the particles are not chosen so that miraculous 
simplifications occur, and the model becomes tractable. In other words, we want our model to have 
``disordered'' hopping rates. Lastly, we would like to prove rigorous statements about the behavior 
of the models. The standard ASM fails the reservoir criterion because grains are usually added to 
all sites, not just at the boundary.  One could force grains to be added only at the 
reservoirs, but the resulting Markov chain could then fail to be ergodic. Even if that were not the case, 
all the dynamics happens because of grains being added externally and there is no intrinsic 
bulk motion, which is what we are interested in studying here.

There are exceptionally few Markov processes which satisfy all the conditions listed above 
and are not one-dimensional. 
Some examples known to us are the Manna model \cite{manna1991,dhar-manna1999},
the stochastic sandpile model \cite{sd2009}, the asymmetric annihilation process \cite{ayyer_strehl_2010}, 
the asymmetric Glauber model \cite{ayyer_2011} and the de-Bruijn process \cite{ayyer_strehl_2013}. 
But these are also one-dimensional. There are very few nontrivial models of nonequilibrium statistical physics 
with reservoirs in higher dimensions where rigorous results are known, but most do 
not seem to be disordered; some examples are given in \cite{dhar2006}. 
We will present results for directed trees, which can be thought of as
quasi two-dimensional objects since they can be embedded in the plane. 

We introduce two new kinds of sandpile models on arbitrary rooted trees,
which we call the \defi{\trickle{}} and the \defi{\landslide}.
Although neither model is abelian, they have beautiful
properties. The stationary distribution of the \trickle{} has a product  form,
which means that the height distributions are independent and there are no correlations.
The \landslide{} has a remarkably simple formula for the
eigenvalues of the transition matrix and their multiplicities.
It also has a fast mixing time (i.e. convergence to steady state) 
which is approximately proportional to the square of
the size of the rooted tree. The underlying basis for this fact, explained below, is
the notion of $\RR$-triviality. This idea is a precise mathematical formulation of 
forgetfulness of the initial distribution for a Markov chain.
Some other examples of ``$\RR$-trivial'' statistical physical models are given
in~\cite{ayyer_strehl_2010,ayyer_2011,ayyer_strehl_2013}. The stationary distribution of the \landslide{} is nontrivial. 
It would be very interesting to calculate physically relevant quantities such as average avalanche sizes, their 
exponents, and various correlation functions.

A novel element of this paper is also our techniques. Influential work
of Diaconis~\cite{DiaconisLectures} and others, going back to the
eighties, has made the character and representation theory of
finite groups extremely relevant to the analysis of Markov chains.  In
groundbreaking work, Bidigare, Hanlon and Rockmore~\cite{BHR}
introduced the new technique of monoid representation theory into the
study of Markov chains and showed how this theory leads to an elegant
analysis of the eigenvalues for Markov chains like the Tsetlin library
and riffle-shuffling.  This approach was further developed by Brown
and Diaconis~\cite{DiaconisBrown}, Brown~\cite{BrownLRB},
Bj\"orner~\cite{bjorner1,bjorner2}, and Chung and
Graham~\cite{GrahamChung}.  The types of monoids used in this theory
are fairly restrictive and Diaconis asked in his 1998 ICM address how
far the monoid techniques can be pushed~\cite{DiaconisICM}.  The third
author initiated a theory for random walks on more general monoids
in~\cite{steinberg.2006,steinberg.2008}.  The first two authors in collaboration with
Klee used these techniques to analyze Markov chains associated to
Sch\"utzenberger's promotion operators on posets~\cite{ayyer_klee_schilling.2012}. 
A multitude of further examples is presented
in~\cite{ayyer_schilling_steinberg_thiery.2013}. The results of these
papers, and this one, rely on the representation theory of the important
class of \defi{$\RR$-trivial monoids}~\cite{EilenbergB}. A key feature of $\RR$-trivial monoids is
that any matrix representation of an $\RR$-trivial monoid can be triangularized.
The point here is that eigenvalues for upper triangular matrices are particularly
easy to compute.

Another new feature in this paper is the use of self-similarity in the
\defi{wreath product} of monoids to analyze Markov chains, and in
particular to compute stationary distributions.  Such techniques have
already been used to great effect for analyzing random walks and the
spectrum of the discrete Laplacian on infinite
groups~\cite{GrigZuk,GNS,KSS}, but they have never before
been used in the monoid context or for finite state Markov chains.

As a side remark, we note that ASMs have also been studied from the
monoid point of view in~\cite{toumpakari2005}.

The paper is organized as follows. In Section~\ref{section.models} we
introduce both variants of the directed nonabelian sandpile model and
state the main results.  Section~\ref{section.monoid sandpile}
reformulates the directed nonabelian sandpile model in terms of wreath
products. In Section~\ref{section.stationary} we present the proofs
for the stationary distributions using the wreath product
approach. For the \trickle{} we provide another proof using a master
equation; it becomes clear that the wreath product approach is
superior in this setting.  Section~\ref{section.Rtrivial} gives a
proof of $\RR$-triviality for the monoid of the \landslide,
which yields the statements about the eigenvalues. We also prove the rate
of convergence.

\subsection*{Acknowledgments}
All the authors would like to thank ICERM, where part of this work was performed, for its hospitality.
AS was partially supported by NSF grants DMS--1001256, OCI--1147247,
and a grant from the Simons Foundation (\#226108 to Anne Schilling).
This work was partially supported by a grant from the Simons Foundation (\#245268
to Benjamin Steinberg).

We would like to thank D. Dhar, F. Bergeron and an anonymous referee for comments.

This research was driven by computer exploration using Maple${}^{{\text TM}}$,
{\sc Sage}~\cite{sage} and {\sc Sage-combinat}~\cite{sage-combinat}.
The {\sc Sage} code is available by request and the Maple package \texttt{NonabelianSandpiles.maple} can be downloaded
from the \texttt{arXiv} source.

\section{Definition of models and statement of results}
\label{section.models}

A tree is a graph without cycles. An \defi{arborescence}, or \defi{out-tree},
is a directed graph
with a special vertex called the \defi{root} such that there is exactly one
directed path from any vertex to the root.\footnote{
Note that many graph theorists prefer the opposite convention
for an arborescence where the path goes
from the root to any vertex (also known as an in-tree) \cite[Section 9.6]{deo1974}.
}
Note that an arborescence on $n$ vertices has exactly $n-1$ directed
edges.  Vertices of degree one in an arborescence are called
\defi{leaves}. An example is given in Figure~\ref{figure.arborescence}.

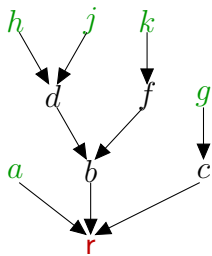
\begin{figure}
\begin{center}
\begin{tikzpicture} [>=triangle 45, x=0.5cm,y=0.5cm]
\draw (1,6) node {{\color{darkgreen}$h$}};
\draw (3,6) node {{\color{darkgreen}$j$}};
\draw (4.5,6) node {{\color{darkgreen}$k$}};
\draw (2,4) node {$d$};
\draw (4.5,4) node {$f$};
\draw (6,4) node {{\color{darkgreen}$g$}};
\draw (1,2) node {{\color{darkgreen}$a$}};
\draw (3,2) node {$b$};
\draw (6,2) node {$c$};
\draw (3,0) node {{\color{darkred}$\rootnode$}};
\draw [->] (1.1,5.7) -- (1.9,4.3);
\draw [->] (2.9,5.7) -- (2.1,4.3);
\draw [->] (4.5,5.7) -- (4.5,4.3);
\draw [->] (2.1,3.7) -- (2.9,2.3);
\draw [->] (4.4,3.7) -- (3.1,2.3);
\draw [->] (6,3.7) -- (6,2.3);
\draw [->] (1.1,1.7) -- (2.9,0.3);
\draw [->] (3,1.7) -- (3,0.3);
\draw [->] (5.9,1.7) -- (3.1,0.3);
\end{tikzpicture}
\caption{An arborescence with {\color{darkgreen} leaves} a, g, h, j, k and
{\color{darkred}root} $\rootnode$
\label{figure.arborescence}}
\end{center}
\end{figure}

We will now informally define our nonabelian sandpile models on
arborescences.  These models will be considered as (discrete-time)
Markov chains. This is a very well-developed theory, see
\cite{MarkovMixing}, for instance. A Markov chain can be thought of as
a random walk on an appropriate graph. For our purposes, we will need
the following facts. If the graph is \defi{strongly connected}, then
the Markov chain is \defi{recurrent}, meaning one can get from any
configuration to any other configuration. If in addition, there is a
single loop in the graph, then the chain is \defi{aperiodic} and it
converges exponentially fast to its unique stationary
distribution. The stationary distribution is, in our convention, the
right eigenvector with eigenvalue 1 of the transition matrix. The
eigenvector is normalized so that the sum of the entries is 1. We
call the normalization factor, albeit with some abuse of terminology,
the \defi{partition function}.

We will define two Markov chains on configurations of arborescences.
In both models, each vertex has a threshold, which is the maximum
number of sand grains that it can accommodate. Sand enters from the
leaves with a certain probability depending on the leaf, one at a
time, flow down the tree, and leave at the root.  Moreover, the
interior vertices can topple with a certain probability. The
difference in the models is in the way in which the interior vertices topple.

In the \trickle, the toppling at a vertex $v$ affects only one grain
of sand at that vertex. That grain moves from $v$, along the directed
path to the root $\rootnode$,
until it finds a vertex $w$ which does not have its threshold number
of sand grains, and settles there. In other words, the number of sand grains at $v$
reduces by 1 and those at $w$ increases by 1. If no such $w$ exists (i.e. all
vertices along the path are filled to capacity), the sand grain exits
the arborescence at the root.

In the \landslide, the toppling at $v$ removes all the grains of sand
at that vertex. These grains are then transferred systematically to the vertices
along the path from $v$ to $\rootnode$. If there are still some grains remaining at
the end, these grains leave the arborescence at the root. Note that if the vertex being
toppled is the root $\rootnode$, then all the sand grains at $\rootnode$ exit the
arborescence.

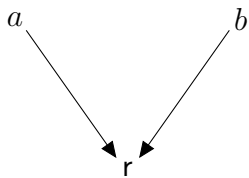
\begin{figure}[h]
\begin{center}
\begin{tikzpicture} [>=triangle 45, x=0.5cm,y=0.5cm]
\draw (0,4) node {$a$};
\draw (6,4) node {$b$};
\draw (3,0) node {$\rootnode$};
\draw [->] (0.3,3.7) -- (2.7,0.3);
\draw [->] (5.7,3.7) -- (3.3,0.3);
\end{tikzpicture}
\caption{Arborescence $\tree_3$ for Example~\ref{example.simple tree}
\label{figure.simple tree}}
\end{center}
\end{figure}

\begin{eg} \label{example.simple tree}
Let $\tree_{3}$ be the arborescence consisting of two leaves and a root,
shown in Figure~\ref{figure.simple tree}, with all thresholds equal to 1.
Note that both the \trickle{} and the \landslide{} are equivalent
in this case. There are 8 states in the Markov chain, which are given by
binary vectors of size 3, denoting the number of grains in vertices $a,b,\rootnode$
in that order. Let the probability for grains entering at vertices $a$ and $b$
be $y_a$ and $y_b$, respectively, and the probability for toppling at the nodes
be $x_a, x_b$ and $x_\rootnode$, with $y_a+y_b+x_a+x_b+x_\rootnode=1$. The graph for the Markov chain is given in Figure~\ref{figure.G tree3}

\begin{figure}
  \begin{bigcenter}
    $\vcenter{
      \includegraphics[width=.6\textwidth]{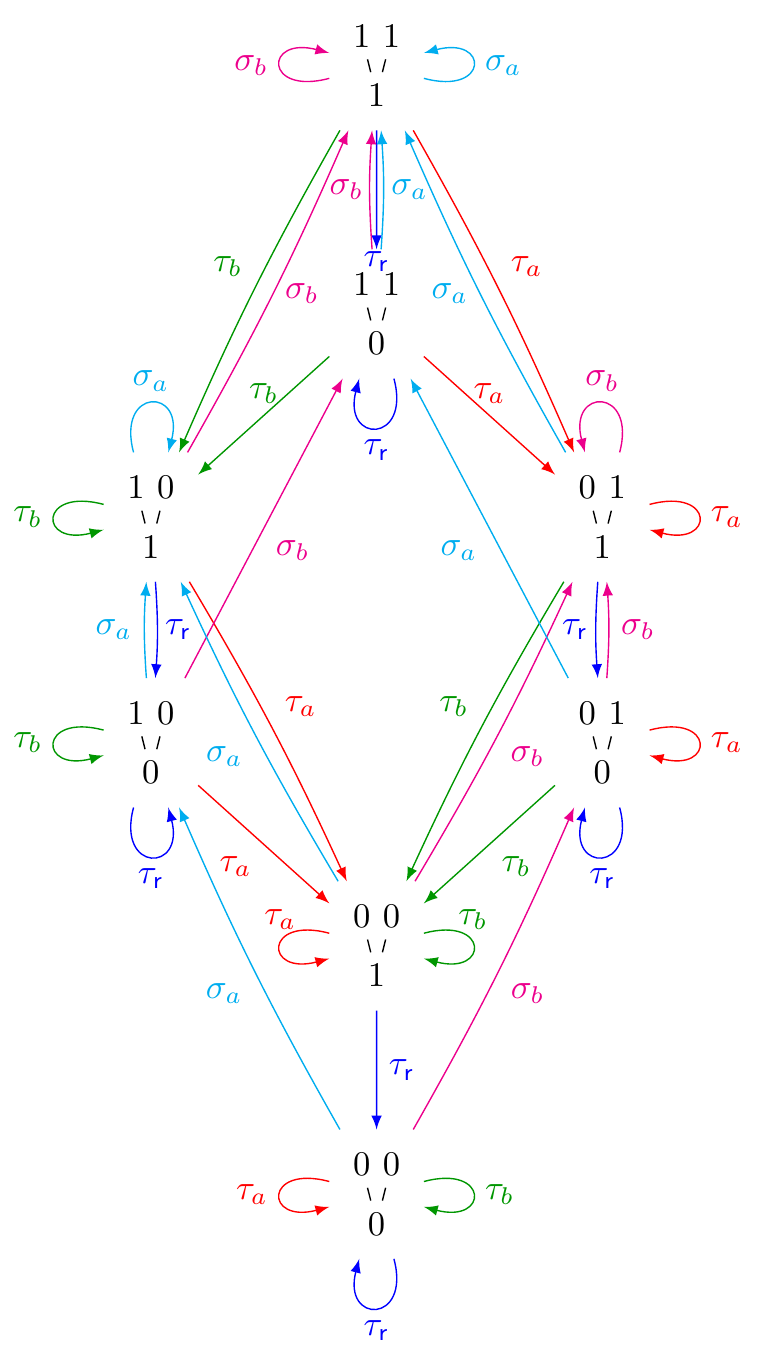}
    }$
    \end{bigcenter}
    \caption{The graph $G_\theta=G_\tau$ of the Markov chain with unit
      thresholds for the arborescence of Example~\ref{example.simple
        tree}.
    \label{figure.G tree3}
  }
\end{figure}

Our convention for the transition matrix for the chain $M$ is that $M_{i,j}$ is
the probability of going from state $j$ to state $i$ so that the
column sums are 1. The rows and columns of $M$ are labelled by the states in
lexicographic order, that is $\{000,001,010,011,100,101,110,111\}$,
\[
M =  \left(
\begin {array}{cccccccc}
*&x_{{{\rootnode}}}&0&0&0&0&0&0\\
\noalign{\medskip}0&*&x_{{b}}&x_{{b}}&x_{{a}}&x_{{a}}&0&0\\
\noalign{\medskip}y_{{b}}&0&*&x_{{{\rootnode}}}&0&0&0&0\\
\noalign{\medskip}0&y_{{b}}&y_{{b}}&*&0&0&x_{{a}}&x_{{a}}\\
\noalign{\medskip}y_{{a}}&0&0&0&*&x_{{{\rootnode}}}&0&0\\
\noalign{\medskip}0&y_{{a}}&0&0&y_{{a}}&*&x_{{b}}&x_{{b}}\\
\noalign{\medskip}0&0&y_{{a}}&0&y_{{b}}&0&*&x_{{{\rootnode}}}\\
\noalign{\medskip}0&0&0&y_{{a}}&0&y_{{b}}&y_{{a}}+y_{{b}}&*
\end {array}
\right),
\]
where the entries on the diagonal are such that column sums are 1. One
can verify that the nonzero entries in $M$ precisely correspond to the
directed arrows is Figure~\ref{figure.G tree3}.

The stationary distribution $\P$ is then the column (right)
eigenvector of $M$ with eigenvalue 1, properly normalized.
The probability for each state is then given by
\begin{align*}
\P(0,0,0)&={\frac {x_{{a}}x_{{b}}x_{{{\rootnode}}}}{ Z }}, \quad
\P(0,0,1)={\frac { x_{{a}}x_{{b}}\left( y_{{a}}+y_{{b}} \right) }{ Z }}, \\
\P(0,1,0)&={\frac {x_{{a}}y_{{b}}x_{{{\rootnode}}}}{ Z }}, \quad
\P(0,1,1)={\frac {x_{{a}} y_{{b}} \left( y_{{a}}+y_{{b}} \right) }{ Z }}, \\
\P(1,0,0)&={\frac {y_{{a}} x_{{b}} x_{{{\rootnode}}}}{ Z }}, \quad
\P(1,0,1)={\frac { y_{{a}}x_{{b}} \left( y_{{a}}+y_{{b}} \right)}{ Z }}, \\
\P(1,1,0)&={\frac {y_{{a}}y_{{b}}x_{{{\rootnode}}}}{ Z }}, \quad
\P(1,1,1)={\frac {y_{{a}} y_{{b}} \left( y_{{a}}+y_{{b}} \right) }{ Z }},
\end{align*}
where $Z$ is the normalization factor, often called the nonequilibrium
partition function,
\[
Z =  \left(x_{{a}}+y_{{a}} \right)  \left( x_{{b}}+y_{{b}} \right) \left( y_{{a}}+y_{{b}}+x_{{{\rootnode}}} \right).
\]
One can see that this is of product form. This property will generalize to the
\trickle{} for all arborescences.

The eigenvalues of $M$ are given by
\begin{align*}
&0, x_{{a}}, x_{{{\rootnode}}}, x_{{b}},
x_{{a}}+x_{{b}}, y_{{a}}+x_{{a}}+x_{{{\rootnode}}},\\ & y_{{b}}+x_{{b}}+x_{{{\rootnode}}},
x_{{b}}+x_{{{\rootnode}}}+y_{{b}}+y_{{a}}+x_{{a}}=1.
\end{align*}
This property of the eigenvalues being partial sums of the probabilities
will persist for the \landslide{} in general.
\end{eg}

The plan for the rest of this section is as follows.
We will first define the state space of our models in Section~\ref{subsection.trees}.
The \trickle{} is defined in Section~\ref{subsection.single grain}, where we also state the
stationary distribution for this model. The \landslide{} is introduced in Section~\ref{subsection.entire}
together with its stationary distribution and precise formulas for the eigenvalues of the transition matrix.
In Section~\ref{subsection.rates} we state the rate of convergence and mixing time for the
\landslide. Finally in Section~\ref{subsection.1dim}, we discuss the specialization of the Markov chains to the case when
the tree is just a one-dimensional line.

\subsection{Arborescences}
\label{subsection.trees}

Let $V$ be the vertex set of the arborescence. We only consider arborescences with finitely
many vertices. The special root vertex is denoted by $\rootnode$.

To each vertex $v\in V$, we associate a \defi{threshold} $T_v$. The state space of our Markov chain
is defined to be
\begin{equation} \label{equation.state space}
	\Omega = \Omega(\tree) := \{(t_v)_{v\in V} \mid 0\le t_v \le T_v\}.
\end{equation}
In other words, at vertex $v$ there can be at most $T_v$ grains. We gather all thresholds in a tuple as
$T=(T_v)_{v\in V}$. The arborescence with its vertices $V$, edges $E$ and thresholds $T$ is denoted
by $\tree=(V,E,T)$. An example of a configuration $t\in \Omega(\tree)$ is given in Figure~\ref{figure.configuration}.

\begin{figure}
\begin{center}
\vspace{-0.5cm}
\begin{tikzpicture} [>=triangle 45, x=0.5cm,y=0.5cm]
\draw (1,6) node {$0$};
\draw (3,6) node {$1$};
\draw (4.5,6) node {$2$};
\draw (2,4) node {$1$};
\draw (4.5,4) node {$0$};
\draw (6,4) node {$0$};
\draw (1,2) node {$1$};
\draw (3,2) node {$2$};
\draw (6,2) node {$1$};
\draw (3,0) node {$1$};
\draw [->] (1.1,5.7) -- (1.9,4.3);
\draw [->] (2.9,5.7) -- (2.1,4.3);
\draw [->] (4.5,5.7) -- (4.5,4.3);
\draw [->] (2.1,3.7) -- (2.9,2.3);
\draw [->] (4.4,3.7) -- (3.1,2.3);
\draw [->] (6,3.7) -- (6,2.3);
\draw [->] (1.1,1.7) -- (2.9,0.3);
\draw [->] (3,1.7) -- (3,0.3);
\draw [->] (5.9,1.7) -- (3.1,0.3);
\end{tikzpicture}
\vspace{-0.5cm}
\caption{A configuration for the arboresence from Figure~\ref{figure.arborescence} with all thresholds $T_v=2$.
\label{figure.configuration}}
\end{center}
\end{figure}
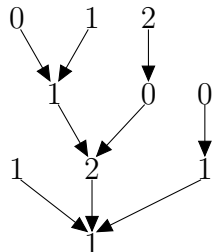

The Markov chain is defined by certain \defi{toppling} and \defi{source operators} on the state space.
We associate a toppling operator $\theta_v\colon \Omega \to \Omega$ to each vertex $v\in V$ for the
\trickle{} (respectively $\tau_v\colon \Omega \to \Omega$ for the \landslide). They topple grains from vertex $v$
along the unique outgoing edge. (Here we assume that an outgoing edge is attached to the root $\rootnode$).
The precise definitions are stated in the next subsections.
In addition, let $L$ be the set of leaves of the arborescence. The source operators
$\source_\ell\colon \Omega\to \Omega$ for $\ell\in L$ are certain operators adding grains at the leaves.

Let $\{x_v,y_\ell \mid v\in V, \ell \in L\}$ be a probability distribution on toppling and source operators,
that is, $x_v$ is the probability of choosing $\theta_v$ (respectively $\tau_v$) and $y_\ell$ is the probability
of choosing $\source_\ell$. We assume that
\begin{enumerate}
\item
$0< x_v, y_\ell\le 1$
\item
$\sum_{v\in V} x_v + \sum_{\ell \in L} y_\ell = 1$
\end{enumerate}
to make it into a proper probability distribution. But in principle these constraints can be relaxed.
This defines for us Markov chains as random walks on graphs whose states are the elements
of $\Omega$ and whose weighted edges are given by the toppling and source operators.

Next we define both models in detail.

\subsection{\trickle{}}
\label{subsection.single grain}

For a vertex $v\in V$ let
\begin{equation}\label{equation.path}
	v^\downarrow=(v= v_0 \to v_1\to \cdots \to v_a=\rootnode)
\end{equation}
be the path from $v$ to the root $\rootnode$; we also use this notation for the set of vertices of the path (the downset of $v$).

\vspace{4mm}
\noindent
{\bf Source operator:} For each leaf $v\in L$, we define a source
operator $\source_v\colon \Omega \to \Omega$ as follows.  As stated
before, the source operator follows the path $v^\downarrow$ from the leaf $v$ to
the root $\rootnode$ and adds a grain to the first vertex along the
way that has not yet reached its threshold, if such a vertex exists.
The precise definition (retaining the notation of \eqref{equation.path}) is: given $t=(t_w)_{w\in V}\in \Omega$, define
$\source_v(t)=t'$ as follows. Let $k\ge 0$ be smallest such that
$t_{v_k}\neq T_{v_k}$. In other words, $t_{v_0}=T_{v_0},\ldots,
t_{v_{k-1}} = T_{v_{k-1}}$, but $t_{v_k} < T_{v_k}$.  Then $t_w'=t_w$
for all $w\in V$ except for $w=v_k$, and $t_{v_k}'=t_{v_k}+1$. If no
such $k$ exists, then $t'=t$.

\vspace{4mm}
\noindent
{\bf Topple operator:} For each vertex $v\in V$, we define a topple operator $\theta_v\colon \Omega\to \Omega$.
Intuitively, $\theta_v$ takes a grain from vertex $v$ and adds it to the first possible site along the path
from $v$ to the root. If there is no available site,
the grain drops out after the root.  Let us give the formal definition.
Let $t=(t_w)_{w\in V}\in \Omega$ and put $\theta_v(t)=t'$ defined as follows.
Consider the path $v^\downarrow$ as in \eqref{equation.path}. If $t_v=0$, then  $\theta_v(t)=t$.
Otherwise $t_v>0$ and let $k\ge 1$ be smallest such that $t_{v_k}\neq T_{v_k}$. In other words,
$t_{v_1}=T_{v_1},\ldots, t_{v_{k-1}} = T_{v_{k-1}}$, but $t_{v_k} < T_{v_k}$.
Then $t_w'=t_w$ except $t_v'=t_v-1$ and $t_{v_k}'=t_{v_k}+1$. If no such $k$ exists, then $t_w'=t_w$
except $t_v'=t_v-1$.

In particular for $v=\rootnode$ the root, we have $t_\rootnode'=\max\{0,t_\rootnode-1\}$ and all other $t_w$ are unchanged.

Examples for the source and topple operators for the \trickle{} are given in Figures~\ref{figure.source operators}
and~\ref{figure.trickle operators}, respectively.

\begin{figure}
\begin{tabular}{c c c c}
$\source_j:$ &
\parbox[c]{4cm}{
\begin{tikzpicture} [>=triangle 45, x=0.5cm,y=0.5cm]
\draw (1,6) node {$0$};
\draw (3,6) node {$2$};
\draw [->] [color=blue] (3,7) -- (3,6.3);
\draw (4.5,6) node {$2$};
\draw (2,4) node {$1$};
\draw (4.5,4) node {$1$};
\draw (6,4) node {$2$};
\draw (1,2) node {$1$};
\draw (3,2) node {$1$};
\draw (6,2) node {$2$};
\draw (3,0) node {$2$};
\draw [->] (1.1,5.7) -- (1.9,4.3);
\draw [->] (2.9,5.7) -- (2.1,4.3);
\draw [->] (4.5,5.7) -- (4.5,4.3);
\draw [->] (2.1,3.7) -- (2.9,2.3);
\draw [->] (4.4,3.7) -- (3.1,2.3);
\draw [->] (6,3.7) -- (6,2.3);
\draw [->] (1.1,1.7) -- (2.9,0.3);
\draw [->] (3,1.7) -- (3,0.3);
\draw [->] (5.9,1.7) -- (3.1,0.3);
\end{tikzpicture}
}
&
$ \mapsto $
&
\parbox[c]{4cm}{
\begin{tikzpicture} [>=triangle 45, x=0.5cm,y=0.5cm]
\draw (1,6) node {$0$};
\draw (3,6) node [color=blue]{$2$};
\draw (4.5,6) node {$2$};
\draw (2,4) node [color=blue]{$2$};
\draw (4.5,4) node {$1$};
\draw (6,4) node {$2$};
\draw (1,2) node {$1$};
\draw (3,2) node {$1$};
\draw (6,2) node {$2$};
\draw (3,0) node {$2$};
\draw [->] (1.1,5.7) -- (1.9,4.3);
\draw [->] (2.9,5.7) -- (2.1,4.3);
\draw [->] (4.5,5.7) -- (4.5,4.3);
\draw [->] (2.1,3.7) -- (2.9,2.3);
\draw [->] (4.4,3.7) -- (3.1,2.3);
\draw [->] (6,3.7) -- (6,2.3);
\draw [->] (1.1,1.7) -- (2.9,0.3);
\draw [->] (3,1.7) -- (3,0.3);
\draw [->] (5.9,1.7) -- (3.1,0.3);
\end{tikzpicture}
}
\end{tabular}

\begin{tabular}{c c c c}
$\source_g:$ &
\parbox[c]{4cm}{
\begin{tikzpicture} [>=triangle 45, x=0.5cm,y=0.5cm]
\draw (1,6) node {$0$};
\draw (3,6) node {$2$};
\draw (4.5,6) node {$2$};
\draw (2,4) node {$1$};
\draw (4.5,4) node {$1$};
\draw (6,4) node [color=blue]{$2$};
\draw [->] [color=blue] (6,5) -- (6,4.3);
\draw (1,2) node {$1$};
\draw (3,2) node {$1$};
\draw (6,2) node {$2$};
\draw (3,0) node {$2$};
\draw [->] (1.1,5.7) -- (1.9,4.3);
\draw [->] (2.9,5.7) -- (2.1,4.3);
\draw [->] (4.5,5.7) -- (4.5,4.3);
\draw [->] (2.1,3.7) -- (2.9,2.3);
\draw [->] (4.4,3.7) -- (3.1,2.3);
\draw [->] (6,3.7) -- (6,2.3);
\draw [->] (1.1,1.7) -- (2.9,0.3);
\draw [->] (3,1.7) -- (3,0.3);
\draw [->] (5.9,1.7) -- (3.1,0.3);
\end{tikzpicture}
}
&
$ \mapsto $
&
\parbox[c]{4cm}{
\begin{tikzpicture} [>=triangle 45, x=0.5cm,y=0.5cm]
\draw (1,6) node {$0$};
\draw (3,6) node {$2$};
\draw (4.5,6) node {$2$};
\draw (2,4) node {$1$};
\draw (4.5,4) node {$1$};
\draw (6,4) node [color=blue]{$2$};
\draw (1,2) node {$1$};
\draw (3,2) node {$1$};
\draw (6,2) node [color=blue]{$2$};
\draw (3,0) node [color=blue]{$2$};
\draw [->] (1.1,5.7) -- (1.9,4.3);
\draw [->] (2.9,5.7) -- (2.1,4.3);
\draw [->] (4.5,5.7) -- (4.5,4.3);
\draw [->] (2.1,3.7) -- (2.9,2.3);
\draw [->] (4.4,3.7) -- (3.1,2.3);
\draw [->] (6,3.7) -- (6,2.3);
\draw [->] (1.1,1.7) -- (2.9,0.3);
\draw [->] (3,1.7) -- (3,0.3);
\draw [->] (5.9,1.7) -- (3.1,0.3);
\end{tikzpicture}
}
\end{tabular}
\caption{Example of source operator actions on states in $\Omega(\tree)$ for $\tree$ as in Figure~\ref{figure.arborescence}.
\label{figure.source operators}}
\end{figure}

\begin{figure}
\begin{tabular}{c c c c}
$\theta_k:$ &
\parbox[c]{4cm}{
\begin{tikzpicture} [>=triangle 45, x=0.5cm,y=0.5cm]
\draw (1,6) node {$0$};
\draw (3,6) node {$2$};
\draw (4.5,6) node [color=blue]{$2$};
\draw (2,4) node {$1$};
\draw (4.5,4) node {$1$};
\draw (6,4) node {$2$};
\draw (1,2) node {$1$};
\draw (3,2) node {$1$};
\draw (6,2) node {$2$};
\draw (3,0) node {$2$};
\draw [->] (1.1,5.7) -- (1.9,4.3);
\draw [->] (2.9,5.7) -- (2.1,4.3);
\draw [->] (4.5,5.7) -- (4.5,4.3);
\draw [->] (2.1,3.7) -- (2.9,2.3);
\draw [->] (4.4,3.7) -- (3.1,2.3);
\draw [->] (6,3.7) -- (6,2.3);
\draw [->] (1.1,1.7) -- (2.9,0.3);
\draw [->] (3,1.7) -- (3,0.3);
\draw [->] (5.9,1.7) -- (3.1,0.3);
\end{tikzpicture}
}
&
$ \mapsto $
&
\parbox[c]{4cm}{
\begin{tikzpicture} [>=triangle 45, x=0.5cm,y=0.5cm]
\draw (1,6) node {$0$};
\draw (3,6) node {$2$};
\draw (4.5,6) node [color=blue]{$1$};
\draw (2,4) node {$2$};
\draw (4.5,4) node [color=blue]{$2$};
\draw (6,4) node {$2$};
\draw (1,2) node {$1$};
\draw (3,2) node {$1$};
\draw (6,2) node {$2$};
\draw (3,0) node {$2$};
\draw [->] (1.1,5.7) -- (1.9,4.3);
\draw [->] (2.9,5.7) -- (2.1,4.3);
\draw [->] (4.5,5.7) -- (4.5,4.3);
\draw [->] (2.1,3.7) -- (2.9,2.3);
\draw [->] (4.4,3.7) -- (3.1,2.3);
\draw [->] (6,3.7) -- (6,2.3);
\draw [->] (1.1,1.7) -- (2.9,0.3);
\draw [->] (3,1.7) -- (3,0.3);
\draw [->] (5.9,1.7) -- (3.1,0.3);
\end{tikzpicture}
}
\end{tabular}

\begin{tabular}{c c c c}
$\theta_g:$ &
\parbox[c]{4cm}{
\begin{tikzpicture} [>=triangle 45, x=0.5cm,y=0.5cm]
\draw (1,6) node {$0$};
\draw (3,6) node {$2$};
\draw (4.5,6) node {$2$};
\draw (2,4) node {$1$};
\draw (4.5,4) node {$1$};
\draw (6,4) node [color=blue]{$2$};
\draw (1,2) node {$1$};
\draw (3,2) node {$1$};
\draw (6,2) node {$2$};
\draw (3,0) node {$2$};
\draw [->] (1.1,5.7) -- (1.9,4.3);
\draw [->] (2.9,5.7) -- (2.1,4.3);
\draw [->] (4.5,5.7) -- (4.5,4.3);
\draw [->] (2.1,3.7) -- (2.9,2.3);
\draw [->] (4.4,3.7) -- (3.1,2.3);
\draw [->] (6,3.7) -- (6,2.3);
\draw [->] (1.1,1.7) -- (2.9,0.3);
\draw [->] (3,1.7) -- (3,0.3);
\draw [->] (5.9,1.7) -- (3.1,0.3);
\end{tikzpicture}
}
&
$ \mapsto $
&
\parbox[c]{4cm}{
\begin{tikzpicture} [>=triangle 45, x=0.5cm,y=0.5cm]
\draw (1,6) node {$0$};
\draw (3,6) node {$2$};
\draw (4.5,6) node {$1$};
\draw (2,4) node {$2$};
\draw (4.5,4) node {$2$};
\draw (6,4) node [color=blue]{$1$};
\draw (1,2) node {$1$};
\draw (3,2) node {$1$};
\draw (6,2) node [color=blue]{$2$};
\draw (3,0) node [color=blue]{$2$};
\draw [->] (1.1,5.7) -- (1.9,4.3);
\draw [->] (2.9,5.7) -- (2.1,4.3);
\draw [->] (4.5,5.7) -- (4.5,4.3);
\draw [->] (2.1,3.7) -- (2.9,2.3);
\draw [->] (4.4,3.7) -- (3.1,2.3);
\draw [->] (6,3.7) -- (6,2.3);
\draw [->] (1.1,1.7) -- (2.9,0.3);
\draw [->] (3,1.7) -- (3,0.3);
\draw [->] (5.9,1.7) -- (3.1,0.3);
\end{tikzpicture}
}
\end{tabular}
\caption{Example of topple operator actions of the \trickle{} on states in $\Omega(\tree)$ for $\tree$
as in Figure~\ref{figure.arborescence}.
\label{figure.trickle operators}}
\end{figure}

\begin{prop} \label{prop:strong1}
The directed graph $G_{\theta}$ whose vertex set is $\Omega$ and whose
edges are given by the operators $\source_\ell$ for $\ell\in L$ and
$\theta_v$ for $v\in V$ is strongly connected and the corresponding
Markov chain is ergodic.
\end{prop}

We defer the proof of Proposition~\ref{prop:strong1} until Section~\ref{subsection.entire}.
Examples of $G_\theta$ are given in Figure~\ref{figure.G tree3} and
Figure~\ref{figure.G theta1d}.

For $v \in V$, let $L_v$ be the set of all sources $\ell\in L$ whose downset $\ell^\downarrow$ contains $v$.
More precisely,
\[
 	L_v := \{ \ell \in L \mid v \in \ell^\downarrow\}.
\]

Moreover, let $Y_v := \sum_{\ell \in L_v} y_\ell$.  For $v \in V$ and $0 \leq h
\leq T_v$, let
\be \label{defrho}
 	\rho_v(h) := \dfrac{Y_v^h \;x_v^{T_v-h}}{\sum_{i=0}^{T_v} Y_v^i\;
x_v^{T_v-i}}.
\ee
Then the following theorem completely describes the stationary distribution.

\begin{thm}\label{stationaryforfirstvarianttree}
The stationary distribution of the Trickle-down sandpile Markov chain
defined on $G_\theta$ is given by the product measure
\be \label{probtree1}
\P(t) = \prod_{v\in V} \rho_v(t_v).
\ee
\end{thm}
A proof of Theorem~\ref{stationaryforfirstvarianttree} using master
equations is given in Section~\ref{subsection.master equation}.  An
alternative proof with an algebraic flavor is presented in
Section~\ref{subsection.stationary}.

The theorem implies that the random variables giving the number of grains at vertex $v$ and at vertex $u$ are independent if we sample from the stationary distribution, regardless of where $u$ and $v$ are located on the tree.

Recall that the (nonequilibrium) partition function of a Markov chain
is the least common denominator of the stationary probabilities. The
following is an immediate corollary of
Theorem~\ref{stationaryforfirstvarianttree}.

\begin{cor} \label{cor:partfn1}
The partition function $Z_\theta$ of the Trickle-down sandpile Markov chain defined on $G_\theta$ is
\[
Z_\theta = \prod_{v \in V} \bigl(\sum_{i=0}^{T_v} Y_v^i x_v^{T_v-i}\bigr).
\]
\end{cor}

\subsection{\landslide{}}
\label{subsection.entire}

In this model, we define the paths $v^\downarrow$ from a vertex $v\in V$ to the root $\rootnode$ as in~\eqref{equation.path}
and the source operators as in Section~\ref{subsection.single grain}. The topple operator on the other hand will topple the
entire site instead of just a single grain.

\vspace{4mm}
\noindent
{\bf Topple operator:} For each vertex $v\in V$, we define a topple
operator $\tau_v\colon \Omega\to \Omega$. As stated before, $\tau_v$ empties
site $v$ and transfers all grains at site $v$ to the first available
sites on the path from $v$ to the root. If there are still grains
remaining, they exit the system from the root.  Formally, we can define $\tau_v=\theta_v^{T_v}$, that is $\tau_v$ is defined to be applying $\theta_v$ as many times as the threshold $T_v$ of $v$.

Examples for the topple operators for the \landslide{} are given in Figure~\ref{figure.landslide operators}.

\begin{figure}
\begin{tabular}{c c c c}
$\tau_k:$ &
\parbox[c]{4cm}{
\begin{tikzpicture} [>=triangle 45, x=0.5cm,y=0.5cm]
\draw (1,6) node {$0$};
\draw (3,6) node {$2$};
\draw (4.5,6) node [color=blue]{$2$};
\draw (2,4) node {$1$};
\draw (4.5,4) node {$1$};
\draw (6,4) node {$2$};
\draw (1,2) node {$1$};
\draw (3,2) node {$1$};
\draw (6,2) node {$2$};
\draw (3,0) node {$2$};
\draw [->] (1.1,5.7) -- (1.9,4.3);
\draw [->] (2.9,5.7) -- (2.1,4.3);
\draw [->] (4.5,5.7) -- (4.5,4.3);
\draw [->] (2.1,3.7) -- (2.9,2.3);
\draw [->] (4.4,3.7) -- (3.1,2.3);
\draw [->] (6,3.7) -- (6,2.3);
\draw [->] (1.1,1.7) -- (2.9,0.3);
\draw [->] (3,1.7) -- (3,0.3);
\draw [->] (5.9,1.7) -- (3.1,0.3);
\end{tikzpicture}
}
&
$ \mapsto $
&
\parbox[c]{4cm}{
\begin{tikzpicture} [>=triangle 45, x=0.5cm,y=0.5cm]
\draw (1,6) node {$0$};
\draw (3,6) node {$2$};
\draw (4.5,6) node [color=blue]{$0$};
\draw (2,4) node {$2$};
\draw (4.5,4) node [color=blue]{$2$};
\draw (6,4) node {$2$};
\draw (1,2) node {$1$};
\draw (3,2) node [color=blue]{$2$};
\draw (6,2) node {$2$};
\draw (3,0) node {$2$};
\draw [->] (1.1,5.7) -- (1.9,4.3);
\draw [->] (2.9,5.7) -- (2.1,4.3);
\draw [->] (4.5,5.7) -- (4.5,4.3);
\draw [->] (2.1,3.7) -- (2.9,2.3);
\draw [->] (4.4,3.7) -- (3.1,2.3);
\draw [->] (6,3.7) -- (6,2.3);
\draw [->] (1.1,1.7) -- (2.9,0.3);
\draw [->] (3,1.7) -- (3,0.3);
\draw [->] (5.9,1.7) -- (3.1,0.3);
\end{tikzpicture}
}
\end{tabular}

\begin{tabular}{c c c c}
$\tau_g:$ &
\parbox[c]{4cm}{
\begin{tikzpicture} [>=triangle 45, x=0.5cm,y=0.5cm]
\draw (1,6) node {$0$};
\draw (3,6) node {$2$};
\draw (4.5,6) node {$2$};
\draw (2,4) node {$1$};
\draw (4.5,4) node {$1$};
\draw (6,4) node [color=blue]{$2$};
\draw (1,2) node {$1$};
\draw (3,2) node {$1$};
\draw (6,2) node {$2$};
\draw (3,0) node {$2$};
\draw [->] (1.1,5.7) -- (1.9,4.3);
\draw [->] (2.9,5.7) -- (2.1,4.3);
\draw [->] (4.5,5.7) -- (4.5,4.3);
\draw [->] (2.1,3.7) -- (2.9,2.3);
\draw [->] (4.4,3.7) -- (3.1,2.3);
\draw [->] (6,3.7) -- (6,2.3);
\draw [->] (1.1,1.7) -- (2.9,0.3);
\draw [->] (3,1.7) -- (3,0.3);
\draw [->] (5.9,1.7) -- (3.1,0.3);
\end{tikzpicture}
}
&
$ \mapsto $
&
\parbox[c]{4cm}{
\begin{tikzpicture} [>=triangle 45, x=0.5cm,y=0.5cm]
\draw (1,6) node {$0$};
\draw (3,6) node {$2$};
\draw (4.5,6) node {$1$};
\draw (2,4) node {$2$};
\draw (4.5,4) node {$2$};
\draw (6,4) node [color=blue]{$0$};
\draw (1,2) node {$1$};
\draw (3,2) node {$1$};
\draw (6,2) node [color=blue]{$2$};
\draw (3,0) node [color=blue]{$2$};
\draw [->] (1.1,5.7) -- (1.9,4.3);
\draw [->] (2.9,5.7) -- (2.1,4.3);
\draw [->] (4.5,5.7) -- (4.5,4.3);
\draw [->] (2.1,3.7) -- (2.9,2.3);
\draw [->] (4.4,3.7) -- (3.1,2.3);
\draw [->] (6,3.7) -- (6,2.3);
\draw [->] (1.1,1.7) -- (2.9,0.3);
\draw [->] (3,1.7) -- (3,0.3);
\draw [->] (5.9,1.7) -- (3.1,0.3);
\end{tikzpicture}
}
\end{tabular}
\caption{Example of topple operator actions of the \landslide{} on states in $\Omega(\tree)$ for $\tree$
as in Figure~\ref{figure.arborescence}.
\label{figure.landslide operators}}
\end{figure}

\begin{rem}
\label{remark.trickle and landslide equivalence}
If the thresholds are all one, that is, $T_v=1$ for all $v\in V$, then the
Trickle-down and Landslide sandpile models are equivalent.
\end{rem}

\begin{rem}
The directed nonabelian sandpile models can also be defined recursively by successively removing leaves.
This approach is taken in Sections~\ref{section.monoid sandpile} through~\ref{section.Rtrivial} when we prove
the stationary distributions and $\RR$-triviality of the underlying monoid.
\end{rem}

\begin{prop} \label{prop:strong2}
The directed graph $G_{\tau}$ whose vertex set is $\Omega$ and whose
edges are given by the operators $\source_\ell$ for $\ell\in L$ and
$\tau_v$ for $v\in V$ is strongly connected and the corresponding
Markov chain is ergodic.
\end{prop}

We prove Propositions~\ref{prop:strong1} and~\ref{prop:strong2} simultaneously.
\begin{proof}
First we prove that $G_{\tau}$ is strongly connected.
By applying the operators $\tau_v$ with $v\in V$ sufficiently often, we can transform
any state to the zero state $(0)_{v\in V}$.

To go from $(0)_{v\in V}$ to any $t\in \Omega$ we use the following strategy.  Let $\ell$ be a
leaf and consider the path \eqref{equation.path}. Suppose that $t'$ satisfies $t'_{v_0}=t'_{v_1}=\cdots=t'_{v_k}=0$.
Then $\tau_{v_{k-1}}\cdots\tau_{v_1}\tau_{v_0}\source_{\ell}^{t_k}t'$ agrees with $t'$ at each vertex except $v_k$,
which will now have $t_k$ grains. Thus applying successively operations of this form, starting with $k=a$, we
can transform $(0)_{v\in V}$ to a vector which agrees with $t$ on all vertices of $\ell^\downarrow$ and has $0$ at all
remaining vertices.  Then proceeding from leaf to leaf, we can eventually reach the vector $t$.  Thus $G_\tau$ is strongly connected.  Observing that $\tau_v=\theta_v^{T_v}$ it immediately follows that $G_\theta$ is also strongly connected.

Both chains are aperiodic because $\tau_v$ and $\theta_v$ fix  $(0)_{v\in V}$ and so both digraphs contain loop edges.
\end{proof}

Since, when all thresholds are one, the \landslide{} is the same as
the \trickle{} (see Remark~\ref{remark.trickle and landslide
  equivalence}), Figure~\ref{figure.G tree3} and Figure~\ref{figure.G
  theta1d} also serve as examples for $G_\tau$.

The stationary distribution in this model is not a product measure  in general.
However, it is in one special case. Let
\begin{equation}\label{definemu}
\mu_v(h) := \begin{cases}
\frac{\ds Y_v^h x_v}{\ds (Y_v + x_v)^{h+1}} & \text{if }h < T_v,\\
\\
\frac{\ds Y_v^{T_v}}{\ds (Y_v + x_v)^{T_v}} & \text{if }h=T_v,
\end{cases}
\end{equation}
where as in Section~\ref{subsection.single grain} we have $Y_v = \sum_{\ell \in L_v} y_\ell$.
It is easy to check that $\ds \sum_{h=1}^{T_v} \mu_v(h)=1$ for all $v\in V$.

\begin{thm} \label{thm.stationarymu}
Let $T_v=1$ for all $v\in V$, $v\neq \rootnode$ and $T_\rootnode = m$ for some positive integer $m$. Then
the stationary distribution of the \landslide{} defined on $G_\tau$ is given by the product measure
\[
\P(t) = \prod_{v\in V} \mu_v(t_v).
\]
\end{thm}

The proof of Theorem~\ref{thm.stationarymu} is given in
Section~\ref{subsection.stationary}. The following is an immediate
consequence of Theorem~\ref{thm.stationarymu}.

\begin{cor} \label{cor:partfn2}
Let $T_v=1$ for all $v\in V$, $v\neq \rootnode$ and $T_\rootnode = m$.
Then the partition function $Z_\tau$ of the \landslide{} defined on
$G_\tau$ is
\[
Z_\tau = \prod_{v \in V} (Y_{v}+x_{v})^{T_{v}}.
\]
\end{cor}

The eigenvalues for the transition matrices for the \landslide{} Markov chain
are given by a very elegant formula. Let $M_{\tau}$ be the transition matrix
for the Markov chain. For $S \subseteq V$, let
\begin{equation} \label{equation.xy}
y_S = \sum_{\ell\in L, \ell^{\downarrow} \subseteq S} y_\ell\;
\text{ and } \;
x_S = \sum_{v \in S} x_v,
\end{equation}
where $\ell^{\downarrow}$ is the set of all vertices on the path from $\ell$ to $\rootnode$.

\begin{thm}\label{charpolysandpile2}
The characteristic polynomial of $M_{\tau}$ is given by
\[
\det(M_{\tau} - \lambda \iden) =
\prod_{S \subseteq V} (\lambda-y_{S} -x_{S})^{T_{S^c}},
\]
where $S^{c} = V\setminus S$ and $T_{S} = \prod_{v \in S} T_v$.
\end{thm}

We defer the proof of this theorem to Section~\ref{subsection.eigenvalues},
where monoid theoretic techniques are used.

\subsection{Rate of convergence for the \landslide}
\label{subsection.rates}

For the \landslide{}, we can make explicit statements about the rate of convergence
to stationarity and mixing times. Let $P^k$ be the distribution after
$k$ steps. The \defi{rate of convergence} is the total variation
distance from stationarity after $k$ steps, that is,
\[||P^k-\pi ||=\frac{1}{2}\sum_{t\in \Omega}|P^k(t)-\pi(t)|\] where $\pi$ is the stationary distribution.

\begin{thm}
  \label{theorem.rate of convergence}
  Define $p_x:=\min\{x_v \suchthat v\in V\}$ and $n:=|V|$. Then, as
  soon as $k\ge (n-1)/p_x$, the distance to stationarity of the
  \landslide{} satisfies
  \begin{displaymath}
    ||P^k-\pi || \leq \exp\left(- \frac{(kp_x-(n-1))^2}{2kp_x}\right)\,.
  \end{displaymath}
\end{thm}

The proof of Theorem~\ref{theorem.rate of convergence} is given in
Section~\ref{subsection.rate of convergence}. Note that the
bound does not depend on the thresholds.

The \defi{mixing time}~\cite{MarkovMixing} is the number of steps $k$ until $||P^k-\pi || \le e^{-c}$
(where different authors use different conventions for the value of $c$). Using
Theorem~\ref{theorem.rate of convergence} we require
\[
(kp_x-(n-1))^2 \ge 2kp_xc\,,
\]
which shows that the mixing time is at
most $\frac{2(n+c-1)}{p_x}$.
If the probability distribution $\{x_v, y_\ell \suchthat v\in V, \ell
\in L\}$ is uniform, then $p_x$ is of order $1/n$ and the mixing time is
of order at most $n^2$.

The above bounds could be further improved.

\subsection{The one-dimensional models}
\label{subsection.1dim}

When the arborescence is a line, both the \trickle{} and the
\landslide{} simplify considerably.
First of all, the notation can be made more concrete.  We may assume that the set of vertices $V$ is
$\{1,\dots,n\}$, which are labeled consecutively from the unique
source $1$ to the root $n$. The threshold vector $T$ is
considered as an $n$-tuple of positive integers.  We denote the
probability of the source operator $\source_{1}$ by $y = y_{1}$ and
the probability of toppling at vertex $i$ by $x_{i}$ in both models
for the sake of consistency. Note that $Y_{i} = y$ for all $i$.
Examples of the one-dimensional models of length 2 and 3 are illustrated
in Figure~\ref{figure.G theta1d}.

The \trickle{} can be thought of as a natural variant of the totally
asymmetric simple exclusion process (TASEP), whose stationary
distribution was computed exactly in \cite{dehp1993}.  In the case
when all the thresholds are equal and all rates 1, this model has been
introduced under the name of the \defi{drop-push process} on the ring
\cite{srb1996}. The model has been generalized to include arbitrary
thresholds and probabilities, but still on the ring \cite{bt1997}. It
is also related to the $m$-TASEP, which has been studied on $\Z$
\cite{sw1998}.  Toppling operations on partitions and compositions
have also been studied from an order-theoretic point of view
\cite{gmp2002}.

\begin{figure}
  \begin{bigcenter}
    $\vcenter{
      \begin{tabular}{c c}
        \raisebox{2cm}{\includegraphics[width=.48\textwidth]{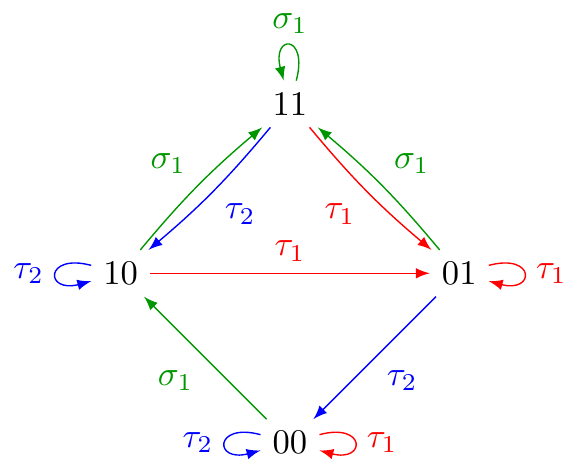}}
        &
        \includegraphics[width=.6\textwidth]{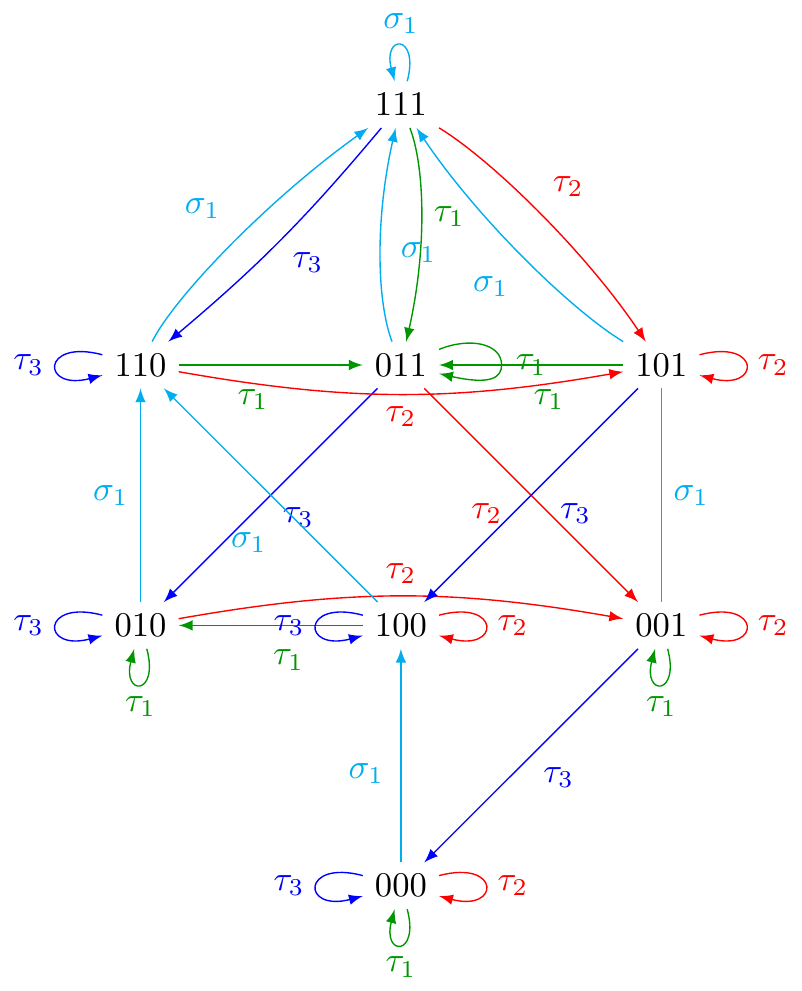}
      \end{tabular}
    }$
    \end{bigcenter}
    \caption{
      The graph $G_\theta=G_\tau$ of the Markov chain with unit
      thresholds for the line $1\rightarrow 2$, the line
      $1\rightarrow 2\rightarrow 3$, respectively.
    \label{figure.G theta1d}
  }
\end{figure}

The stationary distribution for the \trickle{} is a
product measure for all $n$ and any transition probabilities, unlike
for the TASEP.  This follows from
Theorem~\ref{stationaryforfirstvarianttree}.

\begin{cor}\label{cor:stationaryforfirstvariant1d}
The stationary distribution of the Markov chain defined by $G_{\theta}$ is a
product measure,
\[
\P(v) = \prod_{j=1}^n \rho_j(v_j),
\]
where $\rho$ is defined in \eqref{defrho}.
\end{cor}

The \landslide{} is a natural model for the transport of large
self-organizing objects such as macromolecules. These have been
considered in biophysics since at least the 1960s
\cite{ktasep1,ktasep2}.  However, no exact results are known for such
models to the best of our knowledge.

For nonequilibrium statistical systems, it is very rare to have a
concrete example of Markov chains where all the eigenvalues of the
transition matrix are known for any choice of rates. Some examples are
given in \cite{ayyer_strehl_2010,ayyer_2011,ayyer_strehl_2013}. For
models related to abelian sandpiles, there are some conjectures
about eigenvalues in \cite{sd2009, dhar-manna1999}.

\begin{cor}\label{cor:charpolysandpile1d}
The characteristic polynomial of $M_{\tau}$ on $[n]$ is given by
\[
\det(M_{\tau} - \lambda \iden) = (\lambda-x_{0}-x_{[n]})
\prod_{\phi \subseteq S \subsetneq [n]} (\lambda -x_{S})^{T_{S^{c}}},
\]
where $S^{c} = [n]\setminus S$,  $x_{S} = \sum_{i \in S} x_{i}$ and
$T_{S} = \prod_{i \in S} T_{i}$.	
\end{cor}

The stationary distribution of the \landslide{} is not a product
measure, but it still has some interesting structure. The following
conjecture follows from looking at the partition function for various
thresholds vectors up to size 6. We have a similar conjecture for
\landslide{} on trees, but is much more complicated to write down.

\begin{conj} \label{conj:partfn2-1d}
Given the threshold vector $T$ for the \landslide{} on $[n]$, let $k
:= \min\{i \in [n] \suchthat T_i>1\}$.
Then, the partition function $Z_\tau$ is given by
\[
Z_\tau = \prod_{i=1}^{k-1} (y_1 + x_i)^{T_i}
\prod_{\phi \subsetneq S \subseteq \{k,\dots,n\}} (y_1 + x_S)^{T_{\min\{i \suchthat i\in S\}}}.
\]
\end{conj}

One can check that this matches Corollary~\ref{cor:partfn2} when
$k=n$, that is, when the thresholds are one everywhere except at the
root.

\section{Monoids for sandpile models}
\label{section.monoid sandpile}

We will now show how the two variants of the sandpile model can be modeled via the wreath product
of left transformation monoids~\cite{EilenbergB}.  This section is particularly inspired by the theory of
self-similar groups and automaton groups~\cite{GNS,GrigZuk,KSS,selfsimilar}. The wreath product
formulation makes it possible to give a simple proof of the stationary distribution for the \trickle{}
and to prove $\RR$-triviality of the underlying monoid of the \landslide. Using
the results in~\cite{steinberg.2006,steinberg.2008} then yields our results for eigenvalues and multiplicities.

This section is organized as follows. In Section~\ref{subsection.monoid} we review definitions and concepts
from monoid theory that are necessary to prove our theorems. In Section~\ref{subsection.gen wreath}
we present generalities on wreath products. The two variants of the nonabelian sandpile model are reformulated
in Section~\ref{subsection.wreath} in terms of the wreath product.
This formulation will be used in Sections~\ref{section.stationary} and~\ref{section.Rtrivial} to prove
our statements about the stationary distribution, eigenvalues, and rates of convergence.

\subsection{Posets and monoids}
\label{subsection.monoid}

A \defi{partially ordered set} (poset) $P$ is a set with a reflexive, transitive and asymmetric relation $\leq$.
The set of vertices of an arborescence is partially ordered by $v\leq w$ if there is a path from $w$ to $v$.  With this
ordering the root is the smallest element and the leaves are the maximal elements.  An \defi{upset} $U$ in a poset
$P$ is a subset such that $x\in U$ and $y\geq x$ implies $y\in U$.  A \defi{downset} $D$ is
defined dually, $x\in D$ and $y\leq x$ implies $y\in D$.
Denote by $x^{\downarrow}=\{y\in P\mid y\leq x\}$. (This is consistent with the usage in \eqref{equation.path}.) A poset is called a \defi{lattice} if it has a greatest element, a
least element and any two elements have a least upper bound (join) and a greatest lower bound (meet).  For a
finite poset to be a lattice it is enough for it to have a greatest element and binary meets.

A finite \defi{monoid} $M$ is a finite set with an associative
multiplication and an identity element $\id$.
 If $X\subseteq M$, then the submonoid $\langle X\rangle$
generated by $X$ is the smallest submonoid of $M$ containing $X$.  It consists of all (possibly empty) iterated
products of elements of $X$.
Basic references for the theory of monoids are~\cite{clifford_preston.1961,Howie,Higgins}.  Books specializing
in finite monoids are~\cite{EilenbergB,Arbib,Almeida,qtheor,Pin}.

An \defi{action} of a monoid $M$ on a set $\Omega$ is a mapping $M\times \Omega\to \Omega$, written as juxtaposition,
such that $\id x=x$ and $(mm')x=m(m'x)$ for all $m,m'\in M$ and $x\in \Omega$.
Given a probability $\P$ on $M$, we can define a Markov chain $\mathcal M$ on $\Omega$ by defining the
transition probability from $x$ to $y$ to be the probability that $mx=y$ if $m$ is distributed according to $\P$.
The so-called ``random mapping representation'' of a Markov chain~\cite{MarkovMixing} asserts that all finite
state Markov chains can be realized in this way.

The monoid point of view brings a new perspective: the generators of
the monoid act on itself on both sides. This gives rise to the left
(respectively right) \defi{Cayley graph}: its vertex set is $M$, and for
$m,m'\in M$ and $g$ a generator, there is an edge
$m\stackrel{g}{\mapsto} m'$ whenever $m'=gm$ (respectively $m'=mg$). See
Figures~\ref{figure.r_cayley_graph} and~\ref{figure.l_cayley_graph} for examples. Notice that the left
Cayley graph contains the graph of the Markov chain as its lowest
strongly connected component comparing with Figure~\ref{figure.G theta1d}
(a usual feature; the elements of this
component are the constant functions) and is therefore no simpler to
study than the Markov chain itself. On the other hand, the right
Cayley graph is acyclic! This is a strong feature, called
$\RR$-triviality, which we are going to introduce next and
which is used extensively throughout the paper.
As we will show in Section~\ref{subsection.Rtrivial}, the monoids associated
to the \landslide{} introduced on this paper are $\RR$-trivial.

\begin{sidewaysfigure}
 \vspace{10cm}
  \begin{center}
     \hspace{-1.5cm}
  \scalebox{0.5}{\includegraphics{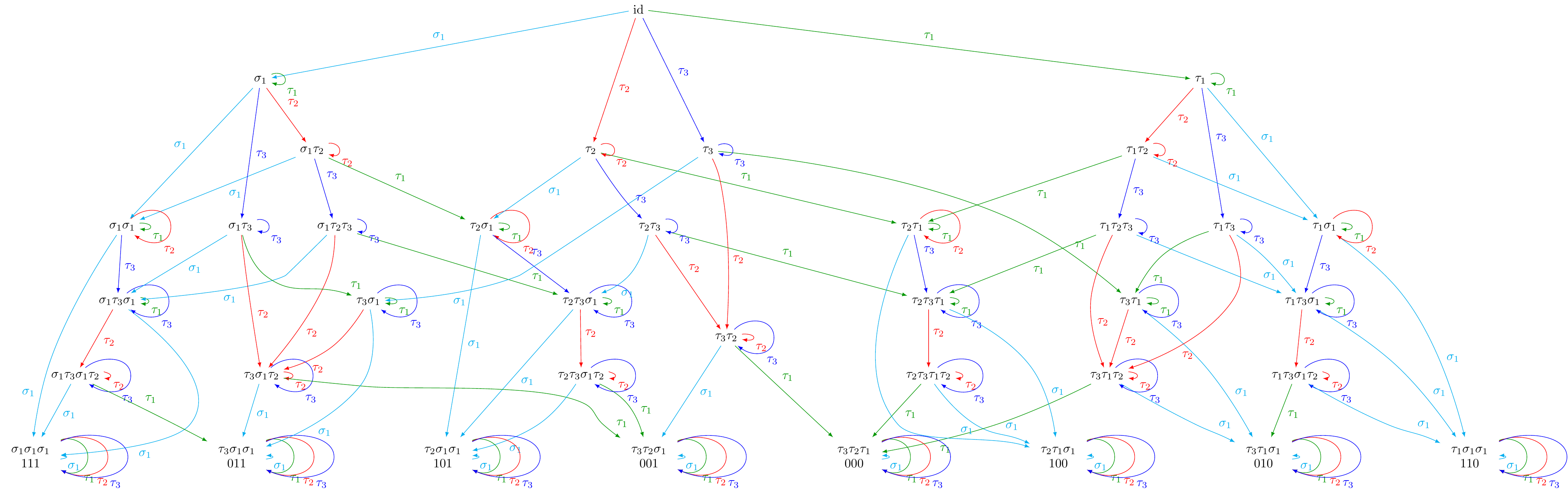}}
  \end{center}
    \caption{The right Cayley graph of the monoid for the one-dimensional
    \landslide{} with three sites as in Figure~\ref{figure.G theta1d}. Each vertex
    displays a reduced word for the corresponding element $m$ of the
    monoid; if $m$ is a constant function, the value of this function is appended.
    \label{figure.r_cayley_graph}
    }
\end{sidewaysfigure}

\begin{figure}
  \includegraphics[width=\textwidth]{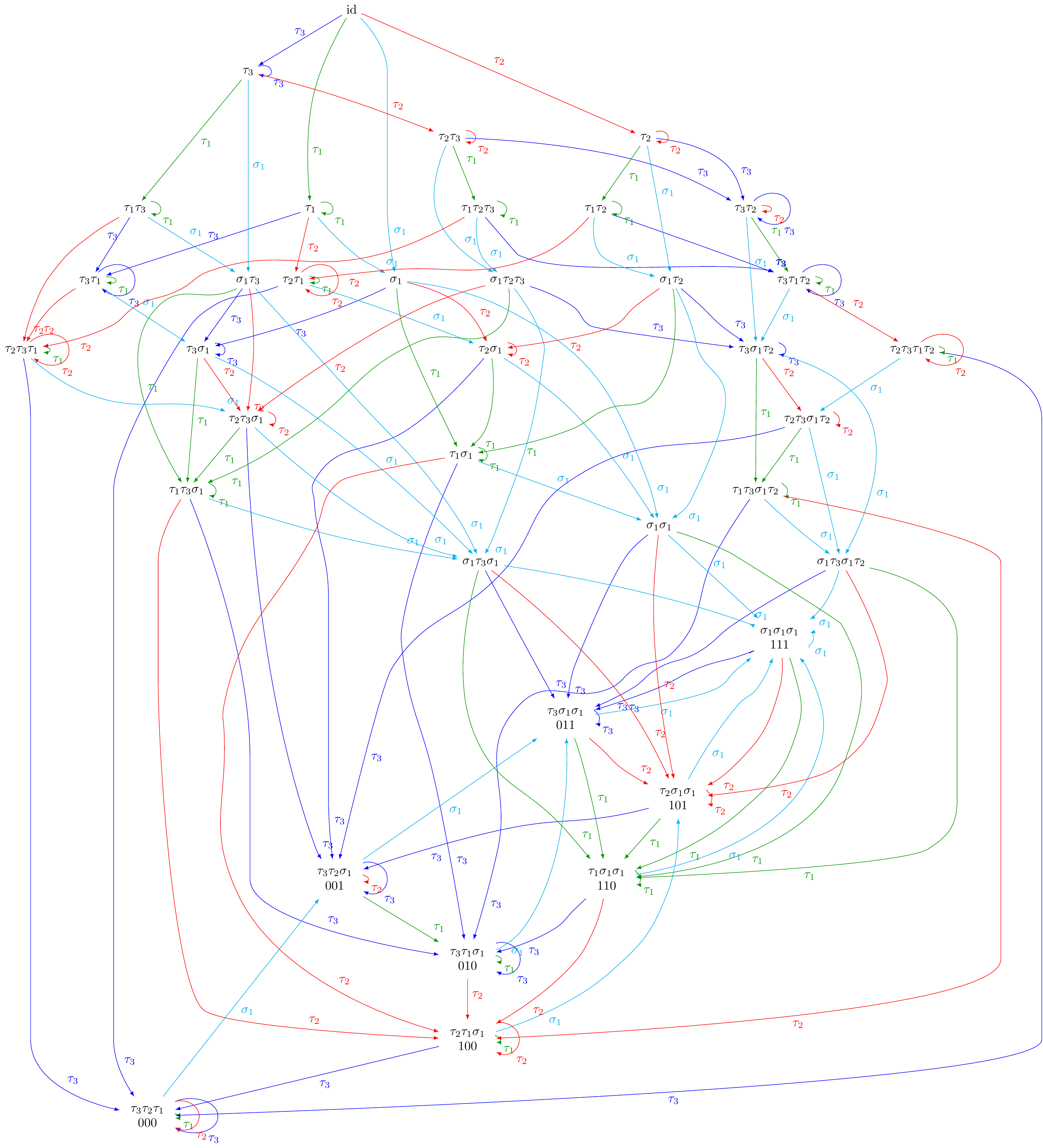}
  \caption{The left Cayley graph of the monoid for the one-dimensional
    \landslide{} with three sites. The constant function
    subgraph agrees with Figure~\ref{figure.G theta1d}.
  \label{figure.l_cayley_graph}
  }
\end{figure}

An element $e$ of a monoid $M$ is called \defi{idempotent} if $e=e^2$. The set of idempotents of $M$ is denoted
$E(M)$.  If $M$ is finite, then each element $m$ has a unique idempotent positive power, traditionally written
$m^{\omega}$. Let $X$ be a generating set for a monoid $M$. Then the \defi{content} of an idempotent $e$ is
defined to be the set $c(e)=\{x\in X\mid e\in MxM\}$.
In other words, $x\in c(e)$ if and only if $e=mxm'$ for some $m,m'\in M$.

A monoid is \defi{$\RR$-trivial} if $aM = bM$ for $a,b\in M$ is equivalent to $a=b$.
Similarly, a monoid is \defi{$\mathscr{J}$-trivial} if $M aM = M bM$ for $a,b\in M$
is equivalent to $a=b$. If $M$ is $\mathscr J$-trivial, then it is well-known
that $e=e'$ if and only if $c(e)=c(e')$ for $e,e'\in E(M)$.  See for instance Chapter~8 of~\cite{Almeida}.
Also a monoid $M$ is $\RR$-trivial if and only if, for each $e\in E(M)$, one has $ex=e$ for all $x\in c(e)$;
see Theorem~5.1 of~\cite{BrzozowskiFich}.  The classes of $\mathscr J$-trivial monoids and $\RR$-trivial
monoids are easily verified to be closed under taking submonoids.

Associated to an $\RR$-trivial monoid is a lattice.  The following can all be extracted from
Chapter~8 of~\cite{Almeida} or Chapter~6 of~\cite{qtheor}, where things are considered in much greater
generality.  A more recent exposition, closer to our viewpoint, can be found in~\cite{MargolisSteinbergQuiver}.
We say that two idempotents $e,f\in M$ are \defi{$\LL$-equivalent}  if $ef=e$ and $fe=f$.  The equivalence
class of $e$ will be denoted by $[e]$.  The set $\Lambda(M)$ of equivalence classes of idempotents is a lattice where
the order is given by $[e]\leq [f]$ if and only if $ef=e$.  The largest element of $\Lambda(M)$ is the class of the
identity and the meet of two elements $e,f$ is $[(ef)^{\omega}]$.

The following theorem is a special case of the results of the third author~\cite{steinberg.2006,steinberg.2008}.
\begin{thm}\label{Rtrivialwalk}
Let $M$ be an $\RR$-trivial monoid acting on a set $\Omega$.  Let $\P$ be a probability distribution on $M$ and
let $\mathcal M$ be the Markov chain with state set $\Omega$, where the transition probability from $x$ to $y$ is the
probability that $mx=y$ (if $m$ is chosen according to $\P$). Then the transition matrix has an eigenvalue
\[\lambda_{[e]}=\sum_{[m^{\omega}]\geq [e]} \P(m)\]
for each $[e]\in \Lambda(M)$.  The multiplicities $\boldsymbol{m}_{[e]}$, $[e]\in \Lambda(M)$, are determined recursively by
the equation \[|e\Omega| = \sum_{[f]\leq [e]} \boldsymbol{m}_{[f]}.\]
\end{thm}

\subsection{Wreath products}
\label{subsection.gen wreath}
We refer to Eilenberg~\cite{EilenbergB} for the wreath product of left
transformation monoids (except he uses right transformation
monoids). Another reference is the book~\cite{Meldrum}. Let
$[n]=\{0,\ldots,n\}$.  If $M$ is a monoid acting faithfully on the
left of $[n]$ and $N$ is a monoid acting faithfully on the left of
$X$, then $(M,[n]) \wr (N,X)$ is the monoid $W$ acting faithfully on
$[n]\times X$ defined as follows.   An element $f\in W$ is of the
form
\begin{equation}\label{wreath.element}
	f=\wreath_f(f_0,\ldots, f_n)\;,
\end{equation}
where $\wreath_f\in M$ and $f_i\in N$, for $0\leq i\leq n$.  The product is given by
\[
	\wreath_f(f_0,\ldots, f_n)\cdot \wreath_g(g_0,\ldots, g_n)
	= \wreath_f\wreath_g(f_{\wreath_g(0)}g_0,\ldots,f_{\wreath_g(n)}g_n).
\]
The action of $f$ as in \eqref{wreath.element} on $[n]\times X$ is given by $f(k,x) = (\wreath_f(k),f_k(x))$. If $\wreath_f$ is the identity,
then we just write $f=(f_0,\ldots,f_n)$. The monoid $W$ will be denoted $M \wr N$
when the underlying sets $[n]$ and $X$ are clear.

An alternative representation of this wreath product is via column monomial matrices.
A matrix is column monomial if each column contains exactly one non-zero entry.
Let $\mathscr T_n$ denote the monoid of all self-maps of $[n]$.  Then the wreath product $\mathscr T_n\wr N$
can be identified with the monoid of all column monomial $n\times n$ matrices over $N\cup \{0\}$ with the usual
matrix multiplication. Notice addition is never needed when multiplying column monomial matrices.  The matrix corresponding
to an element $\wreath_f(f_0,\ldots, f_n)$ is the matrix where the unique non-zero entry of column $j$ is $f_j$ and
this element is placed in row $\wreath_f(j)$.  For example if $f=\wreath(f_0,f_1,f_2)$ where (in two-line notation for a function)
\[
	\wreath = \begin{pmatrix} 0 & 1 &2\\ 0 & 0 &2\end{pmatrix},
\]
then the corresponding column monomial matrix is
\[
	f= \begin{bmatrix} f_0 & f_1 & 0\\ 0 & 0& 0\\ 0 & 0& f_2\end{bmatrix}.
\]

If $M$ is a monoid acting on the left of a set $X$, then the associated linear representation
$\rho_X\colon M\to M_{|X|}(\mathbb C)$ is given by
\[
	\rho_X(m)_{ij} = \begin{cases} 1& \text{if}\ mj=i,\\ 0 & \text{else.}\end{cases}
\]
Crucial to this paper is the following
observation.  Consider the Markov chain $\mathcal M$ with state space $X$ where if we are in state $x$, then we
choose an element $m\in M$ with probability $p_m$ and we transition from $x$ to $mx$. Then the transition matrix of
$\mathcal M$ is given by
\[
	\sum_{m\in M}p_m\cdot \rho_X(m).
\]
This is why the representation theory of monoids is potentially useful to analyze Markov chains.

Let $\rho_X\colon N\to M_{|X|}(\mathbb C)$ be the linear representation
associated to the action of $N$ on $X$.  To describe the linear representation
$\rho \colon \mathscr T_n\wr N\to M_{n|X|}(\mathbb C)$ associated to the action of $\mathscr T_n\wr N$ on $[n]\times X$, we should think of $\mathbb C^{n|X|}$ as $\mathbb C^n\otimes \mathbb C^{|X|}$.  Then
$\rho(f)$ is given by the block column monomial matrix obtained by applying $\rho_X$ to each entry of
the column monomial matrix associated to $f$ (where $\rho_X(0)$ is understood to be the $|X|\times |X|$ zero matrix).

\subsection{A wreath product approach to sandpile models}
\label{subsection.wreath}
Let $\tree=(V,E,T)$ be the data for the arborescence as in Section~\ref{subsection.trees}.  Let $\Omega(\tree)$ be the
state space of the Markov chain associated to $\tree$ (see Eq.~\eqref{equation.state space}).  If $\ell\in L$ is a leaf, define
$\nabla_\ell\tree$ to consist of the arborescence obtained by removing the leaf $\ell$ from the vertex set $V$, the outgoing edge
from $\ell$ from the edge set $E$, and $T_\ell$ from the threshold vector $T$. In this section we shall allow an empty arborescence.
Note that $\Omega(\emptyset)$ is a one-element set.

If $w\in V$ is any vertex of our arborescence we can define an operator $\source_w$ on $\Omega(\tree)$ analogously to the
way the source operator in Section~\ref{subsection.entire} was defined for leaves. For the empty arborescence, we interpret
$\source_\emptyset$ by convention to be the identity on $\Omega(\emptyset)$. We can define a \defi{successor} operator on the
vertices of an arborescence by letting $\su(v)$ be the endpoint of the unique edge from $v$.  For convenience, set
$\su(\rootnode)=\emptyset$.

An important role is played in this paper by two families of monoids corresponding to the two variants of the sandpile model.
The monoids associated to $\tree$ are given by
\[N(\tree)=\langle \source_v,\theta_v\mid v\in V\rangle\quad \text{and}\quad
M(\tree)= \langle \source_v,\tau_v\mid v\in V\rangle.\]
Note that $N(\emptyset)=M(\emptyset) = \{\source_{\emptyset}\}$.  Also, since $\tau_v=\theta^{T_v}_v$, it follows that $M(\tree)$ is a submonoid of $N(\tree)$.

The monoids $N(\tree)$ and $M(\tree)$ can be described recursively as follows.  Fix a leaf $\ell\in L$ of $\tree$ and observe
that $\Omega(\tree) = [T_\ell]\times \Omega(\nabla_\ell\tree)$.  We then have the recursions in Table~\ref{tablerecursion},
where $t\in \Omega(\nabla_\ell \tree)$, $t_\ell\in[T_\ell]$, and the operators on the right hand side are viewed as mappings
on $\Omega(\nabla_v\tree)$.

\begin{table}[tbhp]
\begin{align*}
\source_\ell(t_\ell,t) &= \begin{cases}(t_\ell+1,t) & \text{if}\ t_\ell<T_\ell\\ (T_\ell,\source_{\su(\ell)}t) & \text{if}\ t_\ell=T_\ell\end{cases}\\
\source_v(t_\ell,t)&=(t_\ell,\source_{v}t) & (v\neq \ell)\\
\theta_\ell(t_\ell,t)&=\begin{cases} (t_\ell-1,\source_{\su(\ell)}t) & \text{if}\ t_\ell>0\\ (0,t) & \text{if}\ t_\ell=0\end{cases}\\
\theta_v(t_\ell,t)&=(t_\ell,\theta_{v}t) & (v\neq \ell)\\
\tau_\ell(t_\ell,t)&= (0,\source_{\su(\ell)}^{t_\ell}t)\\
\tau_v(t_\ell,t)&=(t_\ell,\tau_{v}t) & (v\neq \ell).
\end{align*}
\caption{Recursions for sandpile operators\label{tablerecursion}}
\end{table}

To rephrase the recursions in Table~\ref{tablerecursion} in the language of wreath products, we need to introduce some notation.
For an $m\geq 0$, define mappings $\alpha_m$ and $\beta_m$ on $[m]$ as follows:
\begin{align*}
\alpha_m(h) &= \begin{cases}h+1 & \text{if}\ h<m,\\ m & \text{if}\ h=m,\end{cases}\\
\beta_m(h) &= \begin{cases}h-1 & \text{if}\ h>0,\\ 0 & \text{if}\ h=0.\end{cases}
\end{align*}
Denote by $\ov k$ the constant mapping on $[m]$ with image $k$.
Let $N(m) = \langle \alpha_m,\beta_m\rangle$ and $M(m) = \langle \alpha_m,\ov 0\rangle$. Note that $M(m)\subseteq N(m)$.

Clearly
\begin{align*}
N(\tree)&\subseteq (N(T_\ell),[T_\ell])\wr (N(\nabla_\ell \tree), \Omega(\nabla_\ell\tree))\;,\\
M(\tree)&\subseteq (M(T_\ell),[T_\ell])\wr (M(\nabla_\ell \tree), \Omega(\nabla_\ell\tree))\;.
\end{align*}

Indeed, we have
\begin{equation}\label{wreathcoord}
\begin{split}
\source_\ell &= \alpha_{T_\ell}(\id,\ldots,\id,\source_{\su(\ell)})\\
\source_v &= (\source_v,\ldots,\source_v)\qquad\qquad\qquad\qquad (v\neq \ell)\\
\theta_\ell &= \beta_{T_\ell}(\id,\source_{\su(\ell)},\source_{\su(\ell)},\ldots,\source_{\su(\ell)})\\
\theta_v &= (\theta_v,\ldots,\theta_v)\qquad\qquad\qquad\qquad (v\neq \ell)\\
\tau_\ell &= \ov 0(\id,\source_{\su(\ell)},\source^2_{\su(\ell)},\ldots,\source^{T_\ell}_{\su(\ell)})\\
\tau_v &= (\tau_v,\ldots,\tau_v)\qquad\qquad\qquad\qquad (v\neq \ell).
\end{split}
\end{equation}

The wreath product setting immediately yields a convenient description
of the action of the generators by multiplication on the right.
\begin{rem}
  \label{remark.right_regular_representation_wreath}
  Take $f:=\wreath_f(f_0,\ldots,f_{T_\ell})\in M(\tree)$ and $m\in
  M(\nabla_\ell\tree)$. Then,
  \begin{equation*}
  \begin{split}
    \wreath_f(f_0,\ldots,f_{T_\ell})\ \tau_\ell &= \ov{\wreath_f(0)} (f_0,f_0\source_\ell,\ldots,f_0\source_{\su(\ell)}^{T_\ell})\,,\\
    \wreath_f(f_0,\ldots,f_{T_\ell})\ \source_\ell &= (\wreath_f\circ \alpha_{T_\ell}) (f_1,\ldots,f_{T_\ell+1})\,,\\
    \wreath_f(f_0,\ldots,f_{T_\ell})\ m &= \wreath_f (f_0m,\ldots,f_{T_\ell}m)\,,
  \end{split}
  \end{equation*}
  where, for notational convenience,
  $f_i:=f_{T_\ell}\source_{\su(\ell)}^{i-T_\ell}$ for $i>T_\ell$.
\end{rem}

\section{Stationary distributions}
\label{section.stationary}

In this section we prove the stationary distributions stated in Section~\ref{section.models}.
In Section~\ref{subsection.master equation} we use a master equation approach to prove
Theorem~\ref{stationaryforfirstvarianttree}. Both Theorems~\ref{stationaryforfirstvarianttree}
and~\ref{thm.stationarymu} are proved in Section~\ref{subsection.stationary} using wreath products.

\subsection{Master equation proof}
\label{subsection.master equation}

We recall basic facts about the stationary distribution of a finite Markov chain.
The stationary probability of every configuration $t \in \Omega$ satisfies the
\defi{master equation}
\be \label{mastereq}
\sum_{t' \in \Omega} \text{probability}(t \to t') \;\P(t) =
\sum_{t'' \in \Omega} \text{probability}(t'' \to t) \;\P(t''),
\ee
namely that the total weight of the outgoing transitions of any
configuration is equal to the incoming weight.
In cases where the chain is ergodic, the solution to \eqref{mastereq} is
unique up to an overall scaling factor. This factor is determined by the
fact that the sum of all probabilities is one.
Let us denote by
\begin{align*}
\Out(t) &= \{t'\;| \; \text{probability}(t \to t') \neq 0\} \\
\In(t) &= \{t''\;| \; \text{probability}(t'' \to t) \neq 0\}
\end{align*}
the sets of outgoing and incoming  configurations into $t$.
For reversible Markov chains, $\Out(t) = \In(t)$ and this equation
is satisfied term by term simply by setting $t''=t'$. This is essentially
the definition of a reversible chain.

For nonreversible Markov chains, this is not true. In special cases, the
\defi{pairwise balance} condition~\cite{srb1996} is satisfied which says
that there is an invertible map $\phi\colon \Out(t) \to \In(t)$ so that for every
$t' \in \Out(t)$, $\phi(t) = t''$ satisfies
\be \label{pairbal}
\text{probability}(t \to t') \;\P(t)=\text{probability}(t'' \to t) \;\P(t'').
\ee
Obviously, a necessary condition for this to work is that $|\Out(t)| = |\In(t)|$ for
all $t \in \Omega$.

We will need a variant of the pairwise balance condition, which we describe now.
Suppose $P_{O}(t)$ (respectively $P_{I}(t)$) is a set partition of $\Out(t)$ (respectively $\In(t)$)
of the same cardinality and further that there exists an invertible map
$\Phi \colon P_{O}(t) \to P_{I}(t)$ which satisfies for every $\wp \in P_{O}$,
\be \label{genpairbal}
\sum_{t' \in \wp}\text{probability}(t \to t') \;\P(t)=
\sum_{t'' \in \Phi(\wp)}\text{probability}(t'' \to t) \;\P(t'').
\ee
If this happens for all $t$, then the master equation~\eqref{mastereq}
holds
and we say that \defi{partitioned balance} holds.

\begin{proof}[Proof of Theorem~\ref{stationaryforfirstvarianttree}]
By Proposition~\ref{prop:strong1}, this chain is ergodic, and hence has a unique stationary distribution.
Therefore, we simply need to check that $\P(t)$ given by formula~\eqref{probtree1} satisfies the master
equation \eqref{mastereq} for a generic configuration $t\in \Omega$. Set $|V|=n$.

We will do so by showing that partitioned balance holds
with $n+1$ partitions of $\Out(t)$. $n$ of these partitions correspond
to singletons $\{\theta_{v}(t)\}$ for all $v \in V$. The last partition
is given by the $\{\source_\ell(t)\mid \ell \in L \}$. Clearly, these form a
set partition of $\Out(t)$.

We will first describe $\Phi(\{\theta_{v}(t)\})$ and show that \eqref{genpairbal} is satisfied.
If $t_v=0$, then $\theta_{v}(t)=t$ and we set $\Phi(\{\theta_{v}(t)\})=\{t\}$.
In this case~\eqref{genpairbal} reduces to~\eqref{pairbal}, which is easy to check.

When $t_{v} \neq 0$, we first describe the set $\Phi(\{\theta_{v}(t)\})$ in words. It is the set of all
possible configurations $t''$ which make a transition to $t$ in such a way that the last grain
falls at site $v$. Note that this can happen either by toppling another vertex
or by entering at a leaf. More precisely, a branch is a contiguous set of all vertices in
$\ell^\downarrow \setminus v^\downarrow$ for a given leaf $\ell\in L$, ending in $v$ but not including $v$,
which are filled to the threshold in configuration $t$. Define $B_{t}(v)$ to be the
set of all \defi{branches} for any $\ell\in L$. Then $B_{t}(v) \cap L$ is the set of leaves of
the \defi{filled branches}, which are the branches where all the vertices from the leaf to $v$
are filled to the threshold. Similarly, let
\[
	W_{t}(v) = \{w \in \bigcup_{\ell\in L_v} (\ell^\downarrow \setminus v^\downarrow) \setminus B_{t}(v)
	\mid \text{$w$ adjacent to $B_t(v)$}\}.
\]
In other words, the set of vertices in $W_t(v)$ are those that sit just above an \defi{unfilled branch}.

If $w \in W_{t}(v)$, define $t^{(v,w)}$ as follows.
Let $t^{(v,w)}_{v}= t_{v}-1, t^{(v,w)}_{w}= t_{w}+1$
and  $t^{(v,w)}_{u}= t_{u}$  for all other vertices $u$.
For $\ell\in B_{t}(v) \cap L$, let $t^{(v,\ell)}$ be the configuration such that $t^{(v,\ell)}_{v}= t_{v}-1$ and
$t^{(v,\ell)}_{u}= t_{u}$ for all other vertices $u$.
Then
\[
\Phi(\{\theta_{v}(t)\}) =
\{t^{(v,w)} \mid w \in W_{t}(v) \}
  \cup
\{t^{(v,\ell)} \mid \ell \in B_{t}(v) \cap L \}\;.
\]
To verify \eqref{genpairbal} for this partition, we have to show that
\[
x_{v} \P(t) = \sum_{w \in W_{t}(v)} x_{w} \P(t^{(v,w)})
+ \sum_{\ell \in B_{t}(v) \cap L } y_{\ell} \P(t^{(v,\ell)})\;.
\]
But, dividing by $\frac{x_v}{Y_v}\P(t)$ and using \eqref{probtree1} for the
probabilities, this amounts to showing
\[
Y_{v}  = \sum_{w \in W_{t}(v)} Y_{w}
+ \sum_{\ell \in B_{t}(v) \cap L } y_{\ell},
\]
and this is easy to see from the definition of $Y_v=\sum_{\ell\in L_v}y_\ell$.

So far, we have considered all possible interior topplings in $\Out(t)$ and all possible
transitions which end with a sand grain being deposited in the interior.
We now consider the last partition of $\Out(t)$ given by the action of
the boundary operators $\{\source_\ell(t)\mid \ell \in L \}$. We will show that
the corresponding configurations in $\In(t)$ are those which end with a sand grain
leaving from the root.

As before, we consider all branches of the root
$\widetilde B_{t}(\rootnode)$, this time including the root. Note that if the root is not filled
to the threshold, only $\rootnode$ belongs to this set. We use the same terminology
as in the first half of the proof and let
$\widetilde W_{t}(\rootnode)  \subseteq \bigcup_{\ell\in L} \ell^\downarrow  \setminus
\widetilde B_{t}(\rootnode)$ denote the vertices which have an arrow to a vertex in $\widetilde
B_{t}(\rootnode)$. As before, for $w \in \widetilde W_{t}(\rootnode)$ we define
$t^{(\rootnode,w)}$ to be the configurations with $t^{(\rootnode,w)}_{\rootnode}= t_{\rootnode}-1$, and
$t^{(\rootnode,w)}_{u}= t_{u}$ for all other vertices $u$.
For $\ell \in \widetilde B_{t}(\rootnode) \cap L$, we define $t^{(\rootnode,\ell)}=t$.
The case $t^{(\rootnode,\ell)}$ corresponds to the situation where a path from leaf $\ell$ to
$\rootnode$ is completely filled to the threshold in $t$.
Thus, we have to show
\[
\sum_{\ell \in L} y_\ell \P(t) = \sum_{w \in \widetilde W_{t}(\rootnode)}
x_{w} \P(t^{(\rootnode,w)})
+ \sum_{\ell \in \widetilde B_{t}(\rootnode) \cap L } y_\ell \P(t^{(\rootnode,\ell)}).
\]
Dividing by $\P(t)$ on both sides and using \eqref{probtree1}, we obtain
\[
\sum_{\ell \in L} y_\ell = \sum_{w \in \widetilde W_{t}(\rootnode)} Y_{w}
+ \sum_{\ell \in \widetilde B_{t}(\rootnode) \cap L }  y_\ell,
\]
and this is also true exactly as before.
Therefore we have confirmed that pairwise balance holds for this model
using the probabilities given by \eqref{probtree1}.   This
completes the proof.
\end{proof}

An alternative proof of Theorem~\ref{stationaryforfirstvarianttree} with an algebraic flavor is presented
in Section~\ref{subsection.stationary}.

\subsection{Stationary distributions via wreath products}
\label{subsection.stationary}
We now use the wreath product representation of Section~\ref{subsection.wreath} to give an alternative proof of
Theorem~\ref{stationaryforfirstvarianttree} and a proof of Theorem~\ref{thm.stationarymu}. We use the notation of these
theorems and of Section~\ref{subsection.gen wreath}.

What we shall actually do is prove a more general result.  Suppose that we have operators $\psi_1,\ldots, \psi_r$ acting on a set
$\Omega'$ and a threshold $T$. Let $\Omega=[T]\times \Omega'$.  Set $N=\langle\psi_1,\ldots,\psi_r\rangle$. We define
elements in $(N(T),[T])\wr (N,\Omega')$ as follows:
\begin{align*}
\source_0 &= \alpha_T(\id,\ldots,\id,\psi_1)&\\
\theta_0 &= \beta_T(\id,\psi_1,\ldots,\psi_1)&\\
\Psi_i &= \,\,\,\,\,\,\,(\psi_i,\ldots,\psi_i) & (1\leq i\leq r).
\end{align*}

For example, the operators defining $N(\tree)$ for the \trickle{} are of this form, as are the operators
defining $M(\tree)$ for the \landslide{} if $T_\ell=1$, where $\ell$ is the distinguished leaf considered above.

In column monomial form the above operators are given by
\begin{equation}\label{monomialoperators}
\begin{split}
\source_0 &= \begin{bmatrix}  0&0&\cdots &0 &0\\ I & 0 &\ddots &0 &0\\ 0 & I &\ddots &0&\vdots\\ \vdots &
0& \ddots&0 &0
\\ 0 &0&\cdots &I &\psi_1\end{bmatrix},\qquad
 \theta_0 = \begin{bmatrix}  I&\psi_1& 0&\cdots &\vdots\\ 0 & 0  &\psi_1 &0 &0\\ \vdots & 0 &\ddots
 &\ddots &0\\ 0 & 0& \ddots&0 &\psi_1
\\ 0 &0&\cdots &0 &0\end{bmatrix},\\
\Psi_i&= \begin{bmatrix} \psi_i & 0 & \cdots &0\\ 0 &\psi_i& \ddots&\vdots \\ \vdots & \ddots
&\ddots &0\\ 0&\cdots & 0 &\psi_i\end{bmatrix}\qquad\qquad (1\leq i\leq r).
\end{split}
\end{equation}
where we are thinking of $\mathbb C^{|\Omega|}$ as $\mathbb C^T\otimes \mathbb C^{|\Omega'|}$.

Consider the Markov chain $\mathcal M$ with state set $\Omega$, where $\source_0$ is applied
with probability $y_0$, $\theta_0$ is applied with probability $x_0$, and $\Psi_i$ is applied with probability $z_i$.
Let us also consider the derived Markov chain $\mathcal M'$ with state set $\Omega'$, where the $\psi_i$ are applied
 with probabilities
\begin{equation} \label{equation.z}
	z'_i = \begin{cases} \dfrac{z_1+y_0}{1-x_0} & \text{if}\ i=1,\\
                    \dfrac{z_i}{1-x_0} & \text{if}\ 2\le i\leq r.
       \end{cases}
\end{equation}

\begin{prop}\label{stationaryrecursion}
Suppose that $\P'$ is a stationary distribution for $\mathcal M'$ and define
\[\pi(h) = \dfrac{y_0^hx_0^{T-h}}{\sum_{i=0}^Tx_0^iy_0^{T-i}}.\]  Then the product measure
$\P = \pi\times \P'$ on $\Omega$ is a stationary distribution for $\mathcal M$, that is,
$\P(h,t) = \pi(h)\P'(t)$ for all $(h,t)\in \Omega$ is stationary for $\mathcal M$.
\end{prop}
\begin{proof}
Straightforward computation shows that $\pi$ is a probability distribution on $[0,T]$ and that
\begin{equation}\label{pirec}
\pi(h) = \left(\frac{y_0}{x_0}\right)\pi(h-1)\qquad (1\leq h\leq T).
\end{equation}
Denote by $M$ the transition matrix of $\mathcal M$ and $M'$ the transition matrix of $\mathcal M'$.  Then $M=y_0\source_0+x_0\theta_0+\sum_{i=1}^r z_i\Psi_i$ and $M'=\sum_{i=1}^rz_i'\psi_i$.  It then follows
from~\eqref{monomialoperators} that $M$ has the block tridiagonal form
\[
	M =  \begin{bmatrix}x_0I+A & x_0\psi_1 & & & 0\\ y_0I& A& x_0\psi_1 & \\  & \ddots & \ddots & \ddots & \\
	& &y_0I& A& x_0\psi_1   \\ 0& & & y_0I & y_0\psi_1+A\end{bmatrix}\;,
\]
where \[A=\sum_{i=1}^rz_i\psi_i=(1-x_0)M'-y_0\psi_1\]
by~\eqref{equation.z}.
In block vector form we have $\P = \left[\pi(0)\P',\ldots,\pi(T)\P'\right]^T$.
We then compute by direct matrix multiplication the block form of $M\P$ using that $M'\P'=\P'$ and repeated
application of~\eqref{pirec}:
\begin{align*}
[M\P]_0 &= \pi(0)\left[x_0+(1-x_0)M'-y_0\psi_1\right]\P'+\pi(1)x_0\psi_1\P'\\
&= \pi(0)\P'+\left[\pi(1)x_0-\pi(0)y_0\right]\psi_1\P'\\ &= \pi(0)\P';\\
[M\P]_h &= \pi(h-1)y_0\P'+\pi(h)\left[(1-x_0)M'-y_0\psi_1\right]\P'+\pi(h+1)x_0\psi_1\P'\\
&= \left[\pi(h-1)y_0-\pi(h)x_0\right]\P'+\pi(h)\P'+\left[\pi(h+1)x_0-\pi(h)y_0\right]\psi\P'\\&=\pi(h)\P' \qquad\qquad (\text{for}\ 0<h<T);\\
[M\P]_T &= \pi(T-1)y_0\P'+\pi(T)(1-x_0)M'\P'\\ &= \pi(T)\P'+\left[\pi(T-1)y_0-\pi(T)x_0\right]\P'\\ &= \pi(T)\P'.
\end{align*}
Therefore, $M\P = \P$ and so $\P$ is stationary.
\end{proof}

We can now deduce Theorem~\ref{stationaryforfirstvarianttree}.

\begin{proof}[Proof of Theorem~\ref{stationaryforfirstvarianttree}]
We retain the notation of Section~\ref{subsection.wreath}.  We work by induction on the number of vertices, the theorem being
trivial for no vertices.  Let $\mathcal M$ be the Markov chain with state space $\Omega(\tree)$ where $\source_v$ has
probability $y_v$ for all $v\in V$ and $\theta_v$ has probability $x_v$ for all $v\in V$.  Let \[Y_v = \sum_{w\geq v} y_w.\] Define
$\rho_v$, for $v\in V$, and $\P$ as in \eqref{defrho} and \eqref{probtree1}, respectively (but with the above definition of $Y_v$).
We prove $\P$ is the stationary distribution for $\mathcal M$. Theorem~\ref{stationaryforfirstvarianttree} will then follow by
setting $y_v=0$ if $v$ is not a leaf.

Let $\mathcal M'$ be the Markov chain with state space $\Omega(\nabla_v\tree)$ with probabilities
\[
	\widetilde y_v = \begin{cases}\dfrac{y_{\su(\ell)}+y_{\ell}}{1-x_\ell} & \text{if}\ v=\su(\ell),\\
	\dfrac{y_v}{1-x_\ell} & \text{if}\ v\neq \su(\ell),
	\end{cases}
\]
for the $\source_v$ and probabilities $\widetilde x_v=\frac{x_v}{1-x_{\ell}}$ for the $\theta_v$ with $v\in V\setminus \{\ell\}$.
For $v\in V\setminus \{\ell\}$, let
\[
	\widetilde Y_v =\sum_{w\geq v}\widetilde y_w.
\]
Using that $v<\ell$ if and only if $v\leq \su(\ell)$, we conclude that
\begin{equation}\label{movedtosucc}
\widetilde Y_v = \frac{Y_v}{1-x_\ell}.
\end{equation}
Denote by $\P'$ the stationary distribution of $\mathcal M'$.  Notice that if
\[
	\rho'_v(h) =  \frac{\widetilde Y_v^h \;\widetilde x_v^{T_v-h}}{\sum_{i=0}^{T_v} \widetilde Y_v^i\;
	\widetilde x_v^{T_v-i}},
\]
then $\rho'_v(h)=\rho_v(h)$ by \eqref{movedtosucc}.  By induction, we may assume
that the stationary distribution $\P'$ of $\mathcal M'$ is given by
\[
	\P'(t)=\prod_{v\in V\setminus \{\ell\}}\rho_v'(v) = \prod_{v\in V\setminus \{\ell\}}\rho_v(v).
\]
Proposition~\ref{stationaryrecursion} now implies that
\[
	\P(t)=\prod_{v\in V}\rho_v(v),
\]
as required.
\end{proof}

Next we prove Theorem~\ref{thm.stationarymu}.
\begin{proof}[Proof of Theorem~\ref{thm.stationarymu}]
Notice that if the threshold $T_v$ is $1$, then $\theta_v=\tau_v$.  Therefore, the inductive step of the proof of
Theorem~\ref{thm.stationarymu} proceeds identically to that of the proof of Theorem~\ref{stationaryforfirstvarianttree}.
The difference is only in the base case, which will be when the graph just contains the root.  In that case we just have
state space $[T_{\rootnode}]$ and with probability $y_{\rootnode}$ we increase the number of grains by $1$ to a
threshold of $T_{\rootnode}$ and with probability $x_{\rootnode}$ we go to $0$.  This is precisely the classical winning
streak Markov chain~\cite[Example 4.15]{MarkovMixing} and~\eqref{definemu} is the well-known stationary distribution for that chain.
\end{proof}

\section{$\RR$-triviality, eigenvalues, and rate of convergence for the \landslide}
\label{section.Rtrivial}

In this section we first show that the monoid associated to the \landslide{} is $\RR$-trivial. This is then used
to prove the eigenvalues of the transition matrix of this model and the rate of convergence.

\subsection{$\RR$-triviality of $M(\tree)$}
\label{subsection.Rtrivial}
Let us again fix $\tree$ an arborescence with a threshold vector. We retain the notation of Section~\ref{subsection.wreath}.
In particular, $\ell$ will denote a fixed leaf throughout this subsection.

Define a partial order, called the \defi{dominance order}, on
$\Omega(\tree)$ as follows: $t$ is \defi{dominated by} $t'$, written
$t\dom t'$, if for each vertex $v$ one has \[\sum_{w\geq v}t_w\leq \sum_{w\geq v}t'_w.\]

A mapping $f$ on $\Omega(\tree)$ is \defi{order-preserving} if $f(t)
\dom f(t')$ whenever $t\dom t'$. It is called \defi{decreasing} if $f(t)
\dom t$. The set of all order-preserving and decreasing mappings on a
poset is well-known to be a $\JJ$-trivial monoid. See e.g.~\cite{Pin}
or~\cite{denton_hivert_schilling_thiery.2010} for details. 
\begin{lem}\label{Jtrivial}
 The following hold.
\begin{enumerate}
\item  The mappings $\{\tau_v\mid v\in V\}$ preserve the dominance order
  and are decreasing.  Thus the monoid $J(\tree)=\langle \tau_v\mid
  v\in V\rangle$ is $\mathscr J$-trivial.

\item  The mappings $\{\source_v\mid v\in V\}$ preserve the dominance
  order, commute and are increasing. Thus the commutative monoid
  $\langle \source_v\mid v\in V\rangle$ is $\mathscr J$-trivial.

\item All mappings in the monoid $M(\tree)$ preserve dominance order.
\end{enumerate}
\end{lem}

\begin{proof}
We prove part (1) as the other parts are similar to, or a direct consequence of, part (1).
Since by definition $\tau_v=\theta_v^{T_v}$, it suffices to prove that $\theta_v$ is decreasing and order
preserving. Because $\theta_v$ moves at most one grain lower down in the arborescence, it is clear
that it is a decreasing map. It remains to prove that it is order preserving. First we introduce some notation.

Recall that the \defi{$\zeta$-transform} of $t\in \Omega$ (a vector indexed by $V$) is defined by
$\zeta(t)_v = \sum_{w\geq v}t_w$ for a vertex $v$ and note that $t\dom t'$ if and only if $\zeta(t)_v\leq \zeta(t')_v$ for all vertices $v$.
Suppose now that $t\dom t'$.  If $t'_v=0$, then $\theta_v(t)\dom t\dom t'=\theta_v(t')$ and we are done.  So assume that $t'_v>0$.

Since $\theta_v$ changes a state only in vertices belonging to $v^{\downarrow}$, we have that if $w\nleq v$, then
$\zeta(\theta_v(t))_w=\zeta(t)_w\leq \zeta(t')_w= \zeta(\theta_v(t'))_w$.  Thus it suffices to prove
$\zeta(\theta_v(t))_w\leq \zeta(\theta_v(t'))_w$ if $w\leq v$.

Suppose $w=v$. If $t_v>0$, then clearly $\zeta(\theta_v(t))_v = \zeta(t)_v-1\leq \zeta(t')_v-1=\zeta(\theta_v(t'))_v$. On the other hand, if
$t_v=0$, then $\theta_v(t)=t$ and so (recalling $t'_v>0$)
\[
	\zeta(\theta_v(t))_v =\zeta(t)_v= \sum_{u\gtrdot v} \zeta(t)_u\leq \sum_{u\gtrdot v} \zeta(t')_u\leq \zeta(t')_v-1=\zeta(\theta_v(t'))_v\;,
\]
where $u\gtrdot v$ means $u$ covers $v$ in the ordering on vertices. Thus $\zeta(\theta_v(t))_v\leq \zeta(\theta_v(t'))_v$ in either case.

Next suppose that $w<v$ and that $t'_u=T_u$ for all $w\leq u<v$. Let $Y$ consist of those vertices $y\geq w$ such that $y$ is incomparable with $v$ and let $X$ be the set of minimal elements of $Y$.  Note that each vertex of $Y$ is above a unique vertex of $X$ because we are in a tree.  Also if $x\in X$, then $x\ngeq v$ and so $\zeta(\theta_v(t))_x = \zeta(t)_x \leq \zeta(t')_x = \zeta(\theta_v(t'))_x$. Then
\begin{align*}
	\zeta(\theta_v(t))_w &= \zeta(\theta_v(t))_v + \sum_{w\leq u<v}\theta_v(t)_u +  \sum_{x\in X}\zeta(\theta_v(t))_x\\ &\leq \zeta(\theta_v(t'))_v + \sum_{w\leq u<v}T_u +  \sum_{x\in X}\zeta(\theta_v(t'))_x = \zeta(\theta_v(t'))_w\;.
\end{align*}

Finally, suppose that $w<v$ and that $t'_u<T_u$ for some $u$ with $w\leq u<v$.
Then since $\theta_v(t)\dom t$, we have
\[
	\zeta(\theta_v(t))_w\leq \zeta(t)_w\leq \zeta(t')_w=\zeta(\theta_v(t'))_w\;,
\]
where the last equality follows because the grain of sand that was toppled from $v$ winds up at some vertex $v'$ with $w\leq v'<v$.
This completes the proof that $\theta_v$ (and hence $\tau_v$) is order preserving.
\end{proof}

In order to simultaneously handle a mix between topple and source
operators, we will need to split, according to the circumstances, the
tree $\tree$ into a downset and an upset, and to control the action of
the operators on these two parts in a different fashion.

For an idempotent $e$, define
\begin{equation}\label{equation.L}
  L(e) := \bigcup_{\{v\in V \suchthat \source_v\in c(e)\}} v^{\downarrow}
\end{equation}
and $U(e):=L(e)^c$ (where $X^c$ is the complement of a set $X$). Note
that $L(e)$ is a downset and $U(e)$ is an upset.

On an upset the control comes from order preserving properties.
Namely, for an upset $U\subseteq V$, define the dominance preorder
$\dom_U$ on $\Omega(\tree)$ analogously to the usual dominance order except
we only take into account vertices in $U$; that is, $t\dom_U t'$ if, for all $v\in U$, we have \[\sum_{w\geq v}t_w\leq \sum_{w\geq v}t'_w.\] We write $t\equiv_U t'$ if $t\dom_U t'$ and $t'\dom_U t$, that
is, if $t$ and $t'$ coincide on $U$.
As with the full dominance order, the monoid $M(\tree)$ interacts nicely with this preorder.
The fact that $\equiv_U$ is compatible with the action of $M(\tree)$ means
that $M(\tree)$ embeds in the generalized wreath product indexed by a poset
as discussed in~\cite{meldrum.1995}.

\begin{lem}\label{lemma.congruence}
  Let $U$ be an upset of $V$ in $\tree$.
\begin{enumerate}
\item  The mappings $\{\tau_v\mid v\in V\}$ preserve the dominance preorder
  $\dom_U$ and are decreasing.
\item   The mappings $\{\source_v\mid v\in V\}$ preserve the dominance
  preorder $\dom_U$ and are increasing.
\item  Furthermore, if $v\not\in U$, then $\source_v t\equiv_U t$ and
  $\tau_v t\equiv_U t$ for all $t\in \Omega(\tree)$.
  \end{enumerate}
\end{lem}

On a downset, control comes from the fact that functions involving
repeated source operators tend to be constant below these
sources. This is best expressed in the wreath product setting and we
start with some general remarks, whose proofs are straightforward.

\begin{prop}\label{prop.wreathprodidem}
Suppose that $f=\wreath_f(f_0,\ldots, f_{T_\ell})\in M(\tree)$.  Then $f$ is an idempotent
if and only if $\wreath_f=\id$ and $f_i^2=f_i$ for all $0\leq i\leq T_\ell$ or $\wreath_f=\ov k$
and $f_kf_i=f_i$, for $0\leq i\leq T_\ell$ (and so in particular $f_k^2=f_k$).
\end{prop}
\begin{proof}
If $f$ is an idempotent, then $\wreath_f$ must be an idempotent, and hence by definition of $M(T_\ell)$ either
$\wreath_f=\id$ or $\wreath_f=\ov k$ for some $0\leq k\leq T_\ell$.  In the first case, we have
\[
	f^2=(f_0,\ldots, f_{T_\ell})^2=(f_0^2,\ldots,f_{T_\ell}^2)
\]
and hence $f$ is idempotent if and only if $f_i^2=f_i$ for $0\leq i\leq T_\ell$.  In the second case, we have
\[
	f^2=\ov k(f_0,\ldots, f_{T_\ell})\ov k(f_0,\ldots, f_{T_\ell})= \ov k(f_kf_0,\ldots, f_kf_{T_\ell})
\]
and so $f$ is idempotent if and only if $f_kf_i=f_i$ for all $0\leq i\leq T_\ell$.
\end{proof}

We say that $m\in M(\tree)$ is \defi{constant with value} $k_v\in [T_v]$ at the vertex $v$ if $(mt)_v=k_v$ for all $t\in \Omega(\tree)$.

\begin{rem}\label{rightideal}
If $m\in M(\tree)$ is constant at the vertex $v$ with value $k_v$, then so is $mm'$ for all $m'\in M(\tree)$.
This follows, because putting $t'=m't$, we have $(mm't)_v=(mt')_v = k_v$.
\end{rem}

It will be convenient to describe the property of being constant at a vertex in terms of wreath product coordinates.
Recall that $\ell\in L$ is the chosen fixed leaf.

\begin{rem}\label{rem.constantwreathprod}
Let $m=\wreath_m(m_0,\ldots, m_{T_\ell})\in M(\tree)$.  Then $m$ is constant at $\ell$ with value $k_\ell$ if and only if
$\wreath_m=\ov k_\ell$. If $v\neq \ell$, then $m$ is constant with value $k_v$ at $v$ if and only if each $m_{t_{\ell}}$,
with $0\leq t_\ell\leq T_\ell$, is constant with value $k_v$ at $v$. This is immediate because
$m(t_\ell,t) = (\wreath_m(t_\ell),m_{t_\ell}t)$.
\end{rem}

For the next statement we need the notion of words representing monoid elements.
Let $f\in M(\tree)$. Then an expression $f=g_1\cdots g_r$ in terms of the generators
$g_i\in \{\source_v, \tau_v \suchthat v\in V\}$ is called a \defi{word} representing $f$.
If $w$ is a word in the generators of $M(\tree)$, the corresponding element of $M(\tree)$ will be written in
wreath product coordinates as $\wreath_w(w_0,\ldots, w_{T_\ell})$.

\begin{prop}
  \label{prop.wreath.partial}
  If $w$ is a word with at least $r$ occurrences of $\source_\ell$ with $\ell\in L$, then each $w_i$
  is represented by a word with at least $r+i-T_\ell$ occurrences of $\source_{\su(\ell)}$.
\end{prop}
\begin{proof}
We proceed by induction on the length of the word $w$.  If $w$ is empty, there is nothing to prove.  Assume
that the proposition holds for a word $u$ and that $w=ug$ with $g$ a generator.  Suppose that $u$ has at
least $r$ occurrences of $\source_\ell$.  Then we have three cases.  If $g=\tau_\ell$, then by
Remark~\ref{remark.right_regular_representation_wreath} we have $w_i=u_0\source_{\su(\ell)}^i$ and so
$w_i$ is represented by a word with at least $r+i-T_\ell$ occurrences of $\source_{\su(\ell)}$ by induction.
If $g=\source_\ell$, then Remark~\ref{remark.right_regular_representation_wreath} shows that $w_i = u_{i+1}$,
where $u_{T_\ell+1}=u_{T_\ell}\source_{\su(\ell)}$.  Thus by induction $w_i$ has at least $r+i+1-T_\ell=r+1+i-T_\ell$
occurrences of $\source_{\su(\ell)}$.  Finally, if $g\neq \tau_\ell,\source_\ell$, then $w_i=u_ig$ and the result follows by induction.
\end{proof}

The following technical lemma is the key statement for proving that $M(\tree)$ is $\mathscr R$-trivial.

\begin{lem}\label{constantstuff}
Let $e\in E(M(\tree))$ and let $\source_v\in c(e)$.  Then $e$ is constant at each $w\in v^{\downarrow}$.
\end{lem}
\begin{proof}
The proof proceeds by induction on the number of vertices in the arborescence.  The statement is vacuous when
there are no vertices. Let us write $e=\wreath_e(e_0,\ldots,e_{T_\ell})$.  We distinguish two cases: $v\neq \ell$ and $v=\ell$.

\smallskip

\noindent \textbf{Case 1}: $v\neq \ell$.
Let $w\leq v$. By Proposition~\ref{prop.wreathprodidem}, either $\wreath_e=\id$ or $\wreath_e=\ov k$ for
some $0\leq k\leq T_\ell$.  When $\wreath_e = \id$, $e_i^2=e_i$ for all $i$.  Also, neither $\source_\ell$, nor $\tau_\ell$
can belong to $c(e)$.  A glance at \eqref{wreathcoord} then shows that $e_0=\cdots=e_{T_\ell}=e'$ for some
$e'\in E(M(\nabla_\ell\tree))$ with $c(e')=c(e)$.  Thus, by induction, $e'$ is constant at $w$.  But then $e$ is
constant at $w$ by Remark~\ref{rem.constantwreathprod}.

Next suppose that $\wreath_e=\ov k$ with $0\leq k\leq T_{\ell}$. Then $e_k^2=e_k$ and $e_ke_i=e_i$ for
$0\leq i\leq T_\ell$.  According to Remark~\ref{rem.constantwreathprod}, in order for $e$ to be constant at
$v$ with value $k_v$, we need each $e_i$ to be constant at $v$ with value $k_v$. In light of Remark~\ref{rightideal}
and the equalities $e_ke_i=e_i$, it suffices to show that $e_k$ is constant at $v$ with value $k_v$.  By induction, it
therefore is enough to show that $\source_v\in c(e_k)$. This follows immediately from
Remark~\ref{remark.right_regular_representation_wreath} with $m=\sigma_v$.
\smallskip

\noindent \textbf{Case 2}: $v=\ell$.
Let $w\leq \ell$.  Note that $\wreath_e=\ov k$ for some $0\leq k\leq T_{\ell}$ because $\source_\ell\in c(e)$.
Hence $e_k^2=e_k$ and $e_ke_i=e_i$ for all $0\leq i\leq T_{\ell}$.

If $w=\ell$, then we are done by Remark~\ref{rem.constantwreathprod}, so assume $w<\ell$. Since $e$ is
idempotent and $\source_\ell\in c(e)$, we can represent $e$ by a word with at least $T_\ell+1$ occurrences
of $\source_\ell$. Proposition~\ref{prop.wreath.partial} then yields $\source_{\su(\ell)}\in c(e_k)$.  Then by
induction $e_k$ is constant at $w$. Arguing as in the previous case, we conclude that $e$ is constant at $w$.
\end{proof}

An immediate corollary is the following crucial fact.

\begin{cor}\label{constantconvexhull}
If $e\in E(M(\tree))$, then $e$ is constant at each $w\in L(e)$.
\end{cor}

We are now in position to state and prove the main theorem of this
section.
\begin{thm}\label{isRtrivial}
The monoid $M(\tree)$ is $\RR$-trivial.
\end{thm}
\begin{proof}
  As noted in Section~\ref{subsection.monoid}, it is sufficient to
  take any idempotent $e\in E(M(\tree))$ and $g\in c(e)$ and prove
  that $eg=e$. We will do that by controlling $e$ separately on $L(e)$
  and its complement $U(e)$. Take $t\in \Omega(\tree)$. By
  Corollary~\ref{constantconvexhull} $e$ is constant on $L(e)$ and so
  $egt$ and $et$ coincide on $L(e)$ (cf.~Remark~\ref{rightideal}). Thus, it just remains to
  prove that $egt$ and $et$ also coincide on $U(e)$, that is,
  $egt\equiv_{U(e)} et$. First note that by the definition of $L(e)$, if $\source_v\in c(e)$, then $v\notin U(e)$.

\smallskip

  \noindent \textbf{Case 1}: $g$ is of the form $\source_v$ or $\tau_v$ with $v\notin U(e)$. Lemma~\ref{lemma.congruence}
  then yields $gt\equiv_{U(e)} t$, and therefore $et\equiv_{U(e)} et$.

\smallskip

  \noindent \textbf{Case 2}: $g$ is of the form $\tau_v$ with $v\in U(e)$. Here we use the same trick as
  in the usual proof of $\JJ$-triviality of a decreasing
  order-preserving monoid of transformations. Namely, using that $\tau_v\in
  c(e)$, write $e=h\tau_v h'$. Note that any expression of $h$ and $h'$ as a product of generators contains no
  $\source_u$ with $u\in U(e)$ by the definition of $c(e)$ and $L(e)$~\eqref{equation.L}. This implies by
  Lemma~\ref{lemma.congruence} that $h,h',\tau_v$ preserve $\dom_{U(e)}$ and are decreasing on $U(e)$. Therefore,
  we can conclude with:
  \begin{displaymath}
    et = e^2t = e h\tau_vh't \dom_{U(e)} e\tau_vt \dom_{U(e)} et\,.\qedhere
  \end{displaymath}
\end{proof}

\subsection{Eigenvalues}
\label{subsection.eigenvalues}
Our goal is to compute the eigenvalues for the \landslide{} using Theorem~\ref{Rtrivialwalk}.
To each $S\subseteq V$, we associate an idempotent
  \begin{equation}\label{defineidempotentS}
    e_S:=\left(\prod_{v\in S}\tau_v\right)^{\omega}.
  \end{equation}
  Since $J(\tree)$ is $\mathscr J$-trivial, the resulting idempotent
  is independent of the order in which the product is taken.  Note
  that $e_\emptyset$ is the identity, whereas $e_V$ sends all of
  $\Omega(\tree)$ to the zero vector.

The reader should recall the definition of the lattice associated to an $\RR$-trivial monoid in Section~\ref{subsection.monoid}.
Note that $S_1\subseteq S_2$ if and only if $e_{S_2}e_{S_1} = e_{S_2}$, if and only if $[e_{S_1}]\geq [e_{S_1}]$.  Thus the
lattice $\Lambda(J(\tree))$ associated to $J(\tree)$ is isomorphic to the lattice of subsets of $V$ ordered by reverse inclusion.

\begin{prop}
  The lattice $\Lambda(M(\tree))$ of idempotents of $M(\tree)$
  coincides with that of $J(\tree)$.
\end{prop}
This is an immediate consequence of the following lemma.
\begin{lem}\label{findLequiv}
  Let $e$ be an idempotent of $M(\tree)$, and define
  \begin{displaymath}
    S(e) = L(e)\cup \{v\mid \tau_v\in c(e)\}\,.
  \end{displaymath}
  Then $e$ is $\LL$-equivalent to the idempotent $e_{S(e)}$ of
  $J(\tree)$.
\end{lem}
\begin{proof}
Suppose that $v\in U(e)$.  Then we have that $\source_v\notin c(e)$ by definition of $L(e)$ and that $\tau_v\in c(e)$
if and only if $\tau_v\in c(e_{S(e)})$.  Thus $et\equiv_{U(e)} e_{S(e)}t$ for all $t\in \Omega$ by Lemma~\ref{lemma.congruence}
(consider a word $w$ representing $e$ containing each $\tau_v\in c(e)$ and raise it to a large power; then use the third item
of Lemma~\ref{lemma.congruence}).  Another application of Lemma~\ref{lemma.congruence} then yields
\[e_{S(e)}et\equiv_{U(e)}e_{S(e)}^2t=e_{S(e)}t\equiv_{U(e)} et=e^2t\equiv_{U(e)} ee_{S(e)}t.\]

It thus remains to show that $ee_{S(e)}t$ and $et$ (respectively, $e_{S(e)}et$ and $e_{S(e)}t$) coincide on $L(e)$.
To do this, it suffices by Remark~\ref{rightideal} to verify that $e$ and $e_{S(e)}$ are both constant at each vertex of
$L(e)$.  For $e$, this is the precisely the conclusion of Corollary~\ref{constantconvexhull}.
We claim that $e_{S(e)}$ is constant with value $0$ at each vertex of $L(e)$. Indeed, since $e_{S(e)}$ belongs to the
$\JJ$-trivial monoid $\langle\tau_v\mid v\in V\rangle$, we have that $\tau_v e_{S(e)}=e_{S(e)}$ for all $v\in L(e)$.
As $\tau_v$ is constant with value $0$ at $v$, we conclude that $e_{S(e)}$ is constant with value $0$ at all vertices
of $L(e)$ (cf. Remark~\ref{rightideal}).
\end{proof}

We now have all the ingredients to describe the eigenvalues of the
\landslide. This is achieved by computing the character of $M(\tree)$
acting on $\Omega(\tree)$ -- which boils down to counting fixed points
of idempotents -- and inverting that data using the character table --
which reduces to M\"obius inversion along $\Lambda(M(\tree))$ (i.e.
inclusion-exclusion).

\begin{proof}[Proof of Theorem~\ref{charpolysandpile2}]
Observe that the fixed point set of $e_S$ is
\begin{displaymath}
	e_S\Omega = \{t\in \Omega \mid t_v=0\ \text{for all}\ v\in S\}\,.
\end{displaymath}
For $S\subseteq V$, write $S^c$ for $V\setminus S$. Set
\begin{displaymath}
  \Upsilon_S :=\{t\in \Omega\mid t_v=0\iff v\in S\}\,.
\end{displaymath}
Note that $|\Upsilon_S| = \prod_{v\in S^c}T_v=T_{S^c}$.
Also we have that
\begin{equation}\label{disjointunion}
e_S\Omega = \biguplus_{X\supseteq S} \Upsilon_{X}\,.
\end{equation}

Recall that $\Lambda(M(\tree))$ is isomorphic to the lattice of subsets of $V$ ordered by reverse inclusion via
the mapping $[e_S]\mapsto S$.
If $\ell$ is a leaf, then $S(\source_\ell^{\omega}) = \ell^{\downarrow}$.  Thus $[\source_\ell^{\omega}]\geq [e_S]$
if and only if $\ell^{\downarrow}\subseteq S$.  On the other hand, $[\tau_v]\geq [e_S]$ if and only if $v\in S$.
It then follows from Theorem~\ref{Rtrivialwalk}
that to each subset $S\subseteq V$ there is an associated eigenvalue $\lambda_S= x_S+y_S$, where $x_S$ and $y_S$
are defined in~\eqref{equation.xy}.

Let $\boldsymbol{m}_S$ be the multiplicity of $\lambda_S$. As $[e_S]\in \Lambda(M(\tree))$ corresponds to the subset $S$,
Theorem~\ref{Rtrivialwalk} implies
\[
	|e_S\Omega|=\sum_{X\supseteq S} \boldsymbol{m}_{X}.
\]
By \eqref{disjointunion} we have
\[
	|e_S\Omega| = \sum_{X\supseteq S} |\Upsilon_X|=\sum_{X\supseteq S}T_{X^c}
\]
and so M\"obius inversion (see for example~\cite{stanley.ECII})
yields  $\boldsymbol{m}_S=T_{S^c}$ for all $S\subseteq V$, as desired.
\end{proof}

\subsection{Rate of convergence}
\label{subsection.rate of convergence}
In this section we prove Theorem~\ref{theorem.rate of convergence}.
There is a general technique, called coupling from the past, which allows one to bound the distance to stationarity for an
ergodic random walk coming from a monoid action.  Roughly speaking, it says the following. Suppose that we have a
random mapping representation of an ergodic Markov chain $\mathcal M$ with state set $\Omega$ coming from a probability
distribution $\P$ on a monoid $M$ acting on $\Omega$.  Assume furthermore that $M$ contains a constant map.  Then the
distance to stationarity after $k$ steps of $\mathcal M$ is bounded by the probability of not being at a constant map after $k$
steps of the right random walk on $M$ driven by $\P$.  More precisely, we have the following reformulation of~\cite[Theorem~3]{DiaconisBrown}.

\begin{thm}\label{thm.couplingpast}
Let $M$ be a monoid acting on a set $\Omega$ and let $\P$ be a probability distribution on $M$.
Let $\mathcal M$ be the Markov chain with state set $\Omega$ such that the transition probability from $x$ to $y$ is the
probability that $mx=y$ if $m$ is chosen according to $\P$. Assume that $\mathcal M$ is ergodic with stationary
distribution $\pi$ and that some element of $M$ acts as a constant map on $\Omega$.

Letting $P^k$ be the distribution of $\mathcal M$ after $k$ steps and $\P^k$ be the $k^{th}$-convolution power of
$\P$, we have that
\[\
	\|P^k-\pi\|\leq \P^k(M\setminus C)\;,
\]
where $C$ is the set of elements of $M$ acting as constants on $\Omega$.
\end{thm}

In the context of the \landslide{} we let $M$ be the submonoid of $M(\tree)$ generated by $\source_v$ with $v\in L$
and $\tau_v$ with $v\in V$.  We shall define a statistic $u$ on $M$ so that $u(m)=0$ if and only if $m$ is a constant
map. It follows from Theorem~\ref{thm.couplingpast} that $\|P^k-\pi\|$ is bounded by the probability $u(m)>0$.

\begin{proof}[Proof of Theorem~\ref{theorem.rate of convergence}]
Let $m\in M$.  Say that an upset $U$ of vertices is \defi{deterministic} for $m$ if $mt=mt'$ whenever $t\equiv_U t'$
for $t,t'\in \Omega$.
Notice that the set $V$ of all vertices is deterministic for $m$. Also, if $U_1,U_2$ are deterministic for $m$ and
$U_1\cap U_2=U$, then $U$ is also deterministic for $m$.  Indeed, if $t\equiv_U t'$, choose $t''\in \Omega$ such
that $t\equiv_{U_1}t''\equiv_{U_2} t'$.  This can be done because $t$ and $t'$ agree on $U=U_1\cap U_2$. Then
$mt=mt''=mt'$.  It follows that there exists a unique minimum deterministic upset $U(m)$ for $m$.  Moreover, $m$
is constant on $\Omega$ if and only if $U(m)=\emptyset$.  Define the statistic $u$ on $M$ by $u(m)=|U(m)|$.  Then
we have $0\leq u(m)\leq n$ where $n=|V|$, and $u(m)=0$ if and only if $m$ is constant. Note that $u(\id)=n$.

\smallskip
\textbf{Claim 1}: $u$ decreases along $\mathscr R$-order: $u(mm')\leq
  u(m)$ for any $m,m'\in M$. To see this, it suffices to show that $U(m)$ is deterministic for $mm'$, whence
  $U(mm')\subseteq U(m)$. If
  $t\equiv_{U(m)} t'$ then by Lemma~\ref{lemma.congruence}, $m't
  \equiv_{U(m)} m't$ and therefore $mm't=mm't'$.

\smallskip
  \textbf{Claim 2}:  Assume that $v$ is a minimal element of $U(m)$. Then
  $u(m\tau_v)<u(m)$.  It suffices to show that $U'=U(m)\setminus \{v\}$ is deterministic for $m\tau_v$.  If $t\equiv_{U'} t'$,
  then $\tau_vt\equiv_{U'} \tau_vt'$ and
  furthermore $(\tau_vt)_v=0=(\tau_vt')_v$. Therefore,
  $\tau_vt\equiv_{U(m)} \tau_vt'$ and hence $m\tau_v(t)=m\tau_v(t')$.

  \smallskip
   Let us call a step $m_i\mapsto m_{i+1}$ in the random walk on the
  right Cayley graph of $M$ \defi{successful} if either $m_i$ is constant or
  $u(m_{i+1})< u(m_i)$. Claim~1 implies that $u(m_i)=u(m_{i+1})$ if the step is not successful.
  Thus the probability that $u(m)>0$ after $k$ steps of the right random walk on $M$ is the probability of having
  at most $n-1$ successful steps in the first $k$ steps.

  Claim~2 says that each step has
  probability at least $p_x$ to be successful.  Therefore, the probability that $u(m)>0$ after $k$ steps of the right
  random walk on $M$ is bounded above by the probability of having at most $n-1$ successes in $k$ Bernoulli trials
  with success probability $p_x$.

    Using Chernoff's inequality for the cumulative distribution function of a binomial random variable we obtain
 that (see for example~\cite[After Theorem 2.1]{devroye_lugosi.2001})
  \[
  ||P^k-\pi || \le \sum_{i=0}^{n-1} {k\choose i}p_x^i(1-p_x)^{k-i}
  \leq \exp\left(-\frac{(kp_x-(n-1))^2}{2kp_x}\right)\,,
  \]
  where the last inequality holds as long as $k\ge (n-1)/p_x$.
\end{proof}

\bibliographystyle{alpha}
\bibliography{domwals}

\end{document}